\documentclass[reqno]{amsart}
\usepackage{accents}
\usepackage{amsmath}
\usepackage{amssymb,mathrsfs}
\usepackage{amsthm}
\usepackage{amscd}
\usepackage{algorithmic}
\usepackage{algorithm}
\usepackage{array}
\usepackage{bm}
\usepackage{centernot}
\usepackage{comment}
\usepackage{graphicx}
\usepackage{fancyhdr}
\usepackage{hhline}
\usepackage{longtable}
\usepackage{multicol}
\usepackage{multirow}
\usepackage{supertabular}
\usepackage{slashbox}
\usepackage{tabularx}
\usepackage{ulem}
\usepackage{url}
\usepackage{verbatim}

\usepackage{comment}

\newtheorem{cor}{Corollary}

\newtheorem{ex}{Example}
\newtheorem{lem}{Lemma}
\newtheorem{mainthm}{Main Theorem}

\newtheorem*{tamegenprob}{Tame Generators Problem}
\newtheorem{question}{Question}
\newtheorem{rem}{Remark}

\newtheorem{thm}{Theorem}

\makeatletter
\@addtoreset{equation}{section}
\makeatother
\begin{document}

\title[On permutations induced by tame automorphisms]{On permutations induced by tame automorphisms over finite fields}

%%% First Author %%%
\author[K.~Hakuta]{Keisuke Hakuta}
\address{Interdisciplinary Graduate School of Science and Engineering, 
Shimane University, 
1060 Nishikawatsu-cho, Matsue, Shimane 690-8504, Japan}
\email{hakuta(at)cis.shimane-u.ac.jp}
%%%%%%%%%%%%%%%%%%%%%%%%%%%%%%%%%%%%%%%%%%%%%%%%%%%%%%%%%%%%%%%%%%%%%%%

\subjclass[2010]{
Primary: 
14R10; %Algebraic geometry - Affine geometry - Affine spaces (automorphisms, embeddings, exotic structures, cancellation problem)
Secondary: 
12E20, %Field theory and polynomials - General field theory - Finite fields (field-theoretic aspects)
20B25. %Group theory and generalizations - Permutation groups - Finite automorphism groups of algebraic, geometric, or combinatorial structures
}
\keywords{
Affine algebraic geometry, 
polynomial automorphism, 
tame automorphism, 
finite field, 
permutation. 
%multivariate polynomial cryptography. 
}
\date{May 4, 2017.}

\begin{abstract}
The present paper deals with permutations induced by tame automorphisms over finite fields. 
The first main result is a formula for determining the sign of the permutation 
induced by a given elementary automorphism over a finite field. 
The second main result is a formula for determining the sign of the permutation 
induced by a given affine automorphism over a finite field. 
We also give a combining method of the above two formulae to determine the sign of the permutation 
induced by a given triangular automorphism over a finite field. 
As a result, 
for a given tame automorphism over a finite field, 
if we know a decomposition of the tame automorphism 
into a finite number of affine automorphisms and elementary automorphisms, 
then one can easily determine the sign of the permutation induced 
by the tame automorphism. 
%136 words
\end{abstract}
%%%%%%%%%%%%%%%%%%%%%%%%%%%%%%%%%%%%%%%%%%%%%%%%%%%%%%%%%%%%%%%%%%%%%%

\maketitle

%%%%%%%%%%%%%%%%%%%%%%%%%%%%%%%%%%%%%%%%%%%%%%%%%%%%%%%%%%%%%%%%%%%%%%
\section{Introduction}
\label{sec:perm_tame_intro}
%%%%%%%%%%%%%%%%%%%%%%%%%%%%%%%%%%%%%%%%%%%%%%%%%%%%%%%%%%%%%%%%%%%%%%

Let $k$ be a field and let $k[{X_1}, \ldots, {X_n}]$ be a polynomial ring in $n$ indeterminates over $k$. 
An $n$-tuple of polynomials $F = ({f_1}, \ldots, {f_n})$ is called a polynomial map, 
when ${f_i}$ belongs to $k[{X_1}, \ldots, {X_n}]$ for all $i$, $1 \leq i \leq n$. 
A polynomial map $F = ({f_1}, \ldots, {f_n})$ can be viewed as a map from ${k^n}$ to itself 
by defining 
\begin{equation}
F({a_1}, \ldots, {a_n}) = \left({f_1}({a_1}, \ldots, {a_n}), \ldots, {f_n}({a_1}, \ldots, {a_n})\right) 
\label{eqn:perm_tame_intro_poly_map}
\end{equation}
for $({a_1}, \ldots, {a_n}) \in {k^n}$. 
The set of polynomial maps over $k$ and the set of maps from ${k^n}$ to itself 
are denoted by $\mathop{\mathrm{MA}_{n}}(k)$ and $\mathop{\mathrm{Maps}}({k^n}, {k^n})$, respectively. 
Then by \eqref{eqn:perm_tame_intro_poly_map}, we see that there exists a natural map 
\begin{equation*}
\begin{array}{ccccc}
\pi : & \mathop{\mathrm{MA}_{n}}(k) & \rightarrow & \mathop{\mathrm{Maps}}({k^n}, {k^n}). 
\end{array}
\end{equation*}
$\mathop{\mathrm{MA}_{n}}(k)$ (resp. $\mathop{\mathrm{Maps}}({k^n}, {k^n})$) is a semi-group 
with respect to the composition of polynomial maps (resp. maps). 
Moreover, these two semi-groups are monoids with the identity maps as neutral elements, 
and the map $\pi$ is a monoid homomorphism. 

Let $\mathop{\mathrm{GA}_{n}}(k)$ be the subset of invertible elements in $\mathop{\mathrm{MA}_{n}}(k)$. 
An element in $\mathop{\mathrm{GA}_{n}}(k)$ is called a {\it polynomial automorphism}. 
A polynomial automorphism $F = ({f_1}, \ldots, {f_n})$ is said to be {\it affine} 
when $\deg {f_i} = 1$ for $i = 1, \ldots, n$. 
A polynomial map $E_{a_i}$ of the form 
\begin{equation}
\begin{split}
E_{a_{i}} 
&= 
({X_1}, \ldots, {X_{i-1}}, {X_{i}} + {a_i}, {X_{i+1}}, \ldots, {X_n}), \\ %
& 
{a_i} \in k[{X_1}, \ldots, \hat{X_{i}}, \ldots, {X_n}] = k[{X_1}, \ldots, {X_{i-1}}, {X_{i+1}}, \ldots, {X_n}], %
\end{split}
\label{eqn:perm_tame_intro_elementary_poly_auto}
\end{equation}
is also a polynomial automorphism since ${E_{a_{i}}^{-1}} = E_{-a_{i}}$. 
A polynomial automorphism $E_{a_{i}}$ of form 
\eqref{eqn:perm_tame_intro_elementary_poly_auto} is said to be {\it elementary}. 
Let $\mathop{\mathrm{Aff}}_{n}(k)$ denote the set of all affine automorphisms, 
and let $\mathop{\mathrm{EA}}_{n}(k)$ denote the set of all elementary automorphisms. 
$\mathop{\mathrm{Aff}}_{n}(k)$ and $\mathop{\mathrm{EA}}_{n}(k)$ are subgroups of $\mathop{\mathrm{GA}_{n}}(k)$. 
We put $\mathop{\mathrm{TA}_{n}}(k) := \langle \mathop{\mathrm{Aff}}_{n}(k), \mathop{\mathrm{EA}}_{n}(k) \rangle$, 
where $\langle {H_1}, {H_2} \rangle$ is a subgroup of a group $G$ generated by two subgroups ${H_1}, {H_2} \subset G$. 
Then $\mathop{\mathrm{TA}_{n}}(k)$ is also a subgroup of $\mathop{\mathrm{GA}_{n}}(k)$ and is called the tame subgroup. 
If ${\tau} \in \mathop{\mathrm{TA}_{n}}(k)$ then ${\tau}$ is called {\it tame automorphism}, 
and otherwise (i.e., $\tau \in \mathop{\mathrm{GA}_{n}}(k) \setminus \mathop{\mathrm{TA}_{n}}(k)$) 
$\tau$ is called {\it wild automorphism}. 
A polynomial automorphism $J_{a, f}$ of the form 
\begin{equation*}
\begin{split}
J_{a, f} & = %
\left( {a_1}{X_1} + {f_1}({X_2}, \ldots, {X_n}), {a_2}{X_2} %
+ {f_2}({X_3}, \ldots, {X_n}), \ldots, {a_n}{X_n} + {f_n}\right), \\
& 
{a_i} \in k \ \left(i = 1, \ldots, n\right), \ {f_i} \in k[{X_{i+1}}, \ldots, {X_n}] \ %
\left(i = 1, \ldots, n-1\right), \ {f_n} \in k, %
\end{split}
\end{equation*}
is called {\it de Jonqui\`{e}res automorphism} (or {\it triangular automorphism}). 
From the definition, triangular automorphism is trivially polynomial automorphism. 
The set of all elements of the form $J_{a, f}$ is denoted by $\mathop{\mathrm{BA}_{n}}(k)$: 
\begin{equation*}
\begin{split}
\mathop{\mathrm{BA}_{n}}(k) & := %
\{ J_{a, f} \in \mathop{\mathrm{GA}_{n}}(k) \mid %
{a_i} \in k \ \left(i = 1, \ldots, n\right), \\
& 
{f_i} \in k[{X_{i+1}}, \ldots, {X_n}] \ \left(i = 1, \ldots, n-1\right), \ {f_n} \in k \}. 
\end{split}
\end{equation*}
$\mathop{\mathrm{BA}_{n}}(k)$ is a subgroup of $\mathop{\mathrm{GA}_{n}}(k)$. 
$\mathop{\mathrm{BA}_{n}}(k)$ is also a subgroup of $\mathop{\mathrm{TA}_{n}}(k)$, and it is known that 
$\mathop{\mathrm{TA}_{n}}(k) = \langle \mathop{ \mathrm{Aff} }_{n}(k), \mathop{\mathrm{BA}_{n}}(k) \rangle$ %
(cf. \cite[Exercises for \textsection~5.1 -- 1, 2]{Ess00}). 
The Tame Generators Problem asks whether $\mathop{\mathrm{GA}_{n}}(k) = \mathop{\mathrm{TA}_{n}}(k)$, 
and is related to the famous Jacobian conjecture. 

%\medskip
%===========================================================
\begin{tamegenprob}
\label{conjecture:perm_tame_intro_tame_gen_prob}
{\rm 
$\mathop{\mathrm{GA}_{n}}(k) = \mathop{\mathrm{TA}_{n}}(k)$?
}
\end{tamegenprob}
%===========================================================
%\medskip

For any field $k$, we denote the characteristic of the field $k$ by $p = \text{char}(k)$. 
If $k$ is a finite field $\mathbb{F}_{q}$ with $q$ elements 
($p = \text{char}(\mathbb{F}_{q})$, $q = p^m$, and $m \geq 1$), 
we use the symbol ${\pi}_{q}$ instead of ${\pi}$: 
\begin{equation*}
\begin{array}{ccccc}
{\pi}_{q} : & \mathop{\mathrm{MA}_{n}}(\mathbb{F}_{q}) & \rightarrow & \mathop{\mathrm{Maps}}({\mathbb{F}_{q}^n}, {\mathbb{F}_{q}^n}). 
\end{array}
\end{equation*}
The map ${\pi}_{q}$ induces a group homomorphism 
\begin{equation*}
\begin{array}{ccccc}
{\pi}_{q} : & \mathop{\mathrm{GA}_{n}}(\mathbb{F}_{q}) & \rightarrow & \mathop{\mathrm{Sym}}({\mathbb{F}_{q}^n}), 
\end{array}
\end{equation*}
where $\mathop{\mathrm{Sym}}(S)$ is a symmetric group on a finite set $S$. 
Let 
\begin{equation*}
\mathop{\mathrm{sgn}}: \mathop{\mathrm{Sym}}(S) \rightarrow \{ \pm{1} \} %
\end{equation*}
be the sign function. 
The sign function $\mathop{\mathrm{sgn}}$ is a group homomorphism, 
and $\mathop{\mathrm{Ker}}(\mathop{\mathrm{sgn}}) = \mathop{\mathrm{Alt}}(S)$, 
where $\mathop{\mathrm{Alt}}(S)$ is the alternating group on $S$. 
Recall that for any subgroup $G \subseteq \mathop{\mathrm{GA}_{n}}(\mathbb{F}_{q})$, 
${\pi}_{q}\left(G\right)$ is a subgroup of $\mathop{\mathrm{Sym}}({\mathbb{F}_{q}^n})$. 

In the case where $k = \mathbb{F}_{q}$, 
we can consider a slightly different problem from the Tame Generators Problem, 
namely, 
it is natural to investigate the subgroup ${\pi}_{q}\left(G\right)$ of $\mathop{\mathrm{Sym}}({\mathbb{F}_{q}^n})$. 
This problem has first been investigated in the case $G = \mathop{\mathrm{TA}_{n}}(\mathbb{F}_{q})$ by Maubach \cite{Mau01}. 
Indeed, Maubach proved the following theorem (\cite[Theorem~2.3]{Mau01}) and proposed a following problem (\cite[page~3, Problem]{Mau08}): 

%===========================================================
\begin{thm}
\label{thm:perm_tame_intro_poly_auto_Mau01}
{\normalfont\bfseries (\cite[Theorem~2.3]{Mau01})} \ 
If $n \geq 2$, then 
${\pi}_{q}(\mathop{\mathrm{TA}_{n}}(\mathbb{F}_{q})) = \mathop{\mathrm{Sym}}({\mathbb{F}_{q}^n})$ 
if $q$ is odd or $q = 2$. 
If $q = {2^m}$ where $m \geq 2$ then 
${\pi}_{q}(\mathop{\mathrm{TA}_{n}}(\mathbb{F}_{q})) = \mathop{\mathrm{Alt}}({\mathbb{F}_{q}^n})$. 
\end{thm}
%===========================================================

%===========================================================
\begin{question}
\label{quest:perm_tame_intro_question_maubach}
{\normalfont\bfseries (\cite[page~3, Problem]{Mau08})} \ 
{\rm 
For $q = 2^m$ and $m \geq 2$, 
do there exist polynomial automorphisms such that the permutations 
induced by the polynomial automorphisms belong to $\mathop{\mathrm{Sym}}({\mathbb{F}_{q}^n}) \setminus \mathop{\mathrm{Alt}}({\mathbb{F}_{q}^n})$?
}
\end{question}
%===========================================================

If there exists $F \in \mathop{\mathrm{GA}_{n}}(\mathbb{F}_{2^m})$ 
such that ${\pi}_{2^m}(F) \in \mathop{\mathrm{Sym}}({\mathbb{F}_{2^m}^n}) \setminus \mathop{\mathrm{Alt}}({\mathbb{F}_{2^m}^n})$, 
then we must have $F \in \mathop{\mathrm{GA}_{n}}(\mathbb{F}_{2^m}) \setminus \mathop{\mathrm{TA}_{n}}(\mathbb{F}_{2^m})$, 
namely, $F$ is a wild automorphism. 
Thus, Question~\ref{quest:perm_tame_intro_question_maubach} is a quite important problem 
for the Tame Generators Problem in positive characteristic. 
Furthermore, we refer the reader to \cite[Section~1.2]{MR15} for several questions related to \cite[Theorem~2.3]{Mau01}. 
The present paper deals with permutations induced by tame automorphisms over finite fields. 
We address the following questions related to \cite[Theorem~2.3]{Mau01} 
which is a little different from the questions in \cite[Section~1.2]{MR15}. 

%===========================================================
\begin{question}
\label{quest:perm_tame_intro_question_1}
{\rm 
For a given tame automorphism ${\phi} \in \mathop{\mathrm{TA}_{n}}(\mathbb{F}_{q})$, 
how to determine the sign of the permutation induced by ${\phi}$? 
(how to determine $\mathop{\mathrm{sgn}}({\pi}_{q}({\phi}))$?)
}
\end{question}
%===========================================================

%===========================================================
\begin{question}
\label{quest:perm_tame_intro_question_2}
{\rm 
Suppose that $G$ is a subgroup of $\mathop{\mathrm{GA}_{n}}(\mathbb{F}_{q})$. 
What are sufficient conditions on $G$ such that the inclusion relation 
${\pi}_{q}\left(G\right) \subset \mathop{\mathrm{Alt}}({\mathbb{F}_{q}^n})$ holds? 
}
\end{question}
%===========================================================

Question~\ref{quest:perm_tame_intro_question_1} seems to be natural 
since if $q \neq {2^m}$ ($m \geq 2$), one can not determine 
the sign of the permutation induced by a given tame automorphism over a finite field 
from \cite[Theorem~2.3]{Mau01}. 
The information about sign of the permutations might be useful 
for studying the Tame Generators Problem, Question~\ref{quest:perm_tame_intro_question_maubach}, 
or other related questions such as \cite[page~5, Conjecture~4.1]{Mau08}, \cite[page~5, Conjecture~4.2]{Mau08} and so on. 
Question~\ref{quest:perm_tame_intro_question_2} also seems to be natural 
since if $q \neq {2^m}$ ($m \geq 2$) and $G \neq \mathop{\mathrm{TA}_{n}}(\mathbb{F}_{q})$, 
one can not obtain any sufficient condition for the inclusion relation 
${\pi}_{q}\left(G\right) \subset \mathop{\mathrm{Alt}}({\mathbb{F}_{q}^n})$ 
from \cite[Theorem~2.3]{Mau01}. 

\medskip

The contributions of the present paper are as follows. 
The first main result is a formula for determining the sign of the permutation 
induced by a given elementary automorphism over a finite field 
(Section~\ref{sec:perm_tame_main_elementary}, Main Theorem~\ref{thm:perm_tame_main_elementary_auto_sign}). 
Our method to determine the sign of the permutation 
induced by a given elementary automorphism over a finite field, is based on group theory. 
As a consequence of Main Theorem~\ref{thm:perm_tame_main_elementary_auto_sign}, 
one can derive a sufficient condition for the inclusion relation 
${\pi}_{q}\left({\mathop{\mathrm{EA}}}_{n}\left(\mathbb{F}_{q}\right)\right) \subset \mathop{\mathrm{Alt}}({\mathbb{F}_{q}^n})$ 
(Section~\ref{sec:perm_tame_main_elementary}, Corollary~\ref{cor:perm_tame_main_elementary_auto_sign}). 
The second main result is a formula for determining the sign of the permutation 
induced by a given affine automorphism over a finite field 
(Section~\ref{sec:perm_tame_main_affine}, Main Theorem~\ref{thm:perm_tame_main_affine_auto_sign}). 
Our method to determine the sign of the permutation 
induced by a given affine automorphism over a finite field, is based on linear algebra. 
As a consequence of Main Theorem~\ref{thm:perm_tame_main_affine_auto_sign}, 
one can also derive a sufficient condition for the inclusion relation 
${\pi}_{q}\left({\mathop{\mathrm{Aff}}}_{n}\left(\mathbb{F}_{q}\right)\right) \subset \mathop{\mathrm{Alt}}({\mathbb{F}_{q}^n})$ 
(Section~\ref{sec:perm_tame_main_affine}, Corollary~\ref{cor:perm_tame_main_affine_auto_sign}). 
Section~\ref{sec:perm_tame_main_tame} gives a combining method of 
Main Theorem~\ref{thm:perm_tame_main_elementary_auto_sign} and Main Theorem~\ref{thm:perm_tame_main_affine_auto_sign} 
to determine the sign of the permutation induced by a given triangular automorphism over a finite field 
(Section~\ref{sec:perm_tame_main_tame}, Corollary~\ref{cor:perm_tame_main_triangular_auto_sign}). 
As a result, 
for a given tame automorphism over a finite field, 
if we know a decomposition of the tame automorphism 
into a finite number of affine automorphisms and elementary automorphisms, 
then one can easily determine the sign of the permutation induced 
by the tame automorphism (Section~\ref{sec:perm_tame_main_tame}, Corollary~\ref{cor:perm_tame_main_tame_auto_sign}). 

\medskip

The rest of this paper is organized as follows. 
In Section~\ref{sec:perm_tame_notation}, we fix our notation. 
In Section~\ref{sec:perm_tame_main_elementary}, 
we give a method to determine the sign of the permutation 
induced by a given elementary automorphism over a finite field. 
In Section~\ref{sec:perm_tame_main_affine}, 
we give a method to determine the sign of the permutation 
induced by a given affine automorphism over a finite field. 
Section~\ref{sec:perm_tame_main_tame} deals with 
how to determine the sign of the permutation induced 
by a given triangular automorphism and a given tame automorphism over a finite field. 
%Section~\ref{sec:perm_tame_conclusion} concludes the paper. 

%EOF

%%%%%%%%%%%%%%%%%%%%%%%%%%%%%%%%%%%%%%%%%%%%%%%%%%%%%%%%%%%%%%%%%%%%%%
\section{Notation}
\label{sec:perm_tame_notation}
%%%%%%%%%%%%%%%%%%%%%%%%%%%%%%%%%%%%%%%%%%%%%%%%%%%%%%%%%%%%%%%%%%%%%%

Throughout this paper, we use the following notation. 
For any field $k$, we denote the characteristic of the field $k$ by $p = \text{char}(k)$. 
We denote the multiplicative group of a field $k$ by $k^{*} = k \setminus \{0\}$. 
We use the symbols $\mathbb{Z}$, $\mathbb{C}$, $\mathbb{F}_{q}$ to represent 
the rational integer ring, the field of complex numbers, and a finite field with $q$ elements 
($p = \text{char}(\mathbb{F}_{q})$, $q = p^m$, $m \geq 1$). 
We denote the set of non-negative integers by $\mathbb{Z}_{\geq 0}$. 
For a group $G$ and $g \in G$, 
$\text{ord}_{G}(g)$ is the order of $g$, 
namely, $\text{ord}_{G}(g)$ 
is the smallest non-negative integer $x$ 
that holds $g^x = 1_{G}$, where $1_{G}$ is the identity element in the group $G$. 
We denote by $GL_{n}\left(k\right)$ the set of invertible matrices with entries in $k$. 

The polynomial ring $n$ indeterminates over $k$ is denoted by $k[{X_1}, \ldots, {X_n}]$. 
We denote the polynomial ring $k[{X_1}, \ldots, {X_{i-1}}, {X_{i+1}}, \ldots, {X_n}]$ 
(omit the indeterminate ${X_i}$) 
by $k[{X_1}, \ldots, \hat{X_{i}}, \ldots, {X_n}]$. 
Let $\mathop{\mathrm{MA}_{n}}(k)$ be the set of polynomial maps over $k$. 
$\mathop{\mathrm{MA}_{n}}(k)$ is 
a monoid with respect to the composition of polynomial maps, 
and the neutral elements of the monoid $\mathop{\mathrm{MA}_{n}}(k)$ is the identity map. 
We denote the subset of invertible elements in $\mathop{\mathrm{MA}_{n}}(k)$, 
the set of all affine automorphisms, the set of all elementary automorphisms, 
and the set of all triangular automorphisms 
by $\mathop{\mathrm{GA}_{n}}(k)$, $\mathop{\mathrm{Aff}}_{n}(k)$, $\mathop{\mathrm{EA}}_{n}(k)$, 
and $\mathop{\mathrm{BA}_{n}}(k)$, respectively. 
$\mathop{\mathrm{GA}_{n}}(k)$ is a group with respect to the composition of polynomial automorphisms, 
and $\mathop{\mathrm{Aff}}_{n}(k)$, $\mathop{\mathrm{EA}}_{n}(k)$ are subgroups of $\mathop{\mathrm{GA}_{n}}(k)$. 
Recall that 
\begin{equation}
{\mathop{\mathrm{Aff}}}_{n}\left(k\right) \cong GL_{n}\left(k\right) \ltimes k^{n}. %
\label{eqn:perm_tame_main_aff_auto_semidirect_product}
\end{equation}

For a finite set $S$, we denote the cardinality of $S$ by ${\sharp} S$, 
and denote the symmetric group on $S$ and the alternating group on $S$ 
by $\mathop{\mathrm{Sym}}(S)$ and $\mathop{\mathrm{Alt}}(S)$, respectively. 
Let $S = \{{s_1}, \ldots, {s_n}\}$, $\sharp S = n$, and $\sigma \in \mathop{\mathrm{Sym}}(S)$. 
For $\sigma \in \mathop{\mathrm{Sym}}(S)$, 
\begin{align*}
\sigma %
& = %
\begin{pmatrix}
{s_1}                      & {s_2}                      & \cdots & {s_n} \\
{s_{\sigma\left(1\right)}} & {s_{\sigma\left(2\right)}} & \cdots & {s_{\sigma\left(n\right)}}
\end{pmatrix}
, %
\end{align*}
means that $\sigma\left({s_i}\right) = {s_{\sigma\left(i\right)}}$ for $1 \leq i \leq n$. 
For $\sigma, \tau \in \mathop{\mathrm{Sym}}(S)$, we denote by $\tau \circ \sigma$ the composition 
\begin{align*}
\tau \circ \sigma %
& = %
\begin{pmatrix}
{s_1}                                       & {s_2}                                       & \cdots & {s_n} \\
{s_{\tau\left(\sigma\left(1\right)\right)}} & {s_{\tau\left(\sigma\left(2\right)\right)}} & \cdots & {s_{\tau\left(\sigma\left(n\right)\right)}}
\end{pmatrix}
. %
\end{align*}
The permutation on $S$ defined by 
\begin{equation*}
\begin{split}
\begin{array}{cccl}
          S           & \longrightarrow &           S           & \\
\rotatebox{90}{$\in$} &                 & \rotatebox{90}{$\in$} & \\
s_i          & \longmapsto &   s_i   & \text{if } i \not\in \{{i_1}, \ldots, {i_r}\} \\
s_{i_1}      & \longmapsto & s_{i_2} & \\
             &    \vdots   &         & \\
s_{i_{r-1}}  & \longmapsto & s_{i_r} & \\
s_{i_{r}}    & \longmapsto & s_{i_1} & \\
\end{array}
\end{split}
\end{equation*}
is called the cycle of length $r$ and is denoted 
by $\left({i_1} \; {i_2} \; \ldots \; {i_r}\right)$. 
Let ${\delta}: \mathbb{R} \rightarrow \mathbb{R}$ be a function 
satisfying that ${\delta}\left(x\right) = 0$ when $x = 0$, 
and ${\delta}\left(x\right) = 1$ when $x \neq 0$. 
Let ${\chi}: \mathbb{F}_{q}^{*} \rightarrow \mathbb{C}^{*}$ be a non-trivial multiplicative character of order $\ell$. 
We extend ${\chi}$ to $\mathbb{F}_{q}$ by defining ${\chi}\left(0\right) = 0$. 

%EOF

%%%%%%%%%%%%%%%%%%%%%%%%%%%%%%%%%%%%%%%%%%%%%%%%%%
\section{Sign of permutations induced by elementary automorphisms}
\label{sec:perm_tame_main_elementary}
%%%%%%%%%%%%%%%%%%%%%%%%%%%%%%%%%%%%%%%%%%%%%%%%%%

In this section, we consider the sign of permutations induced by elementary automorphisms over finite fields. 
The main result of this section is as follows. 

%===========================================================
\begin{mainthm}
\label{thm:perm_tame_main_elementary_auto_sign}
{\rm\bfseries (Sign of elementary automorphisms)} \ 
Suppose that $E_{a_{i}}^{(q)}$ 
is an elementary automorphism over a finite field, namely, 
\begin{equation}
E_{a_{i}}^{(q)} = ({X_1}, \ldots, {X_{i-1}}, {X_{i}} + {a_i}, {X_{i+1}}, \ldots, {X_n}) \in {\mathop{\mathrm{EA}}}_{n}\left(\mathbb{F}_{q}\right), %
\label{eqn:perm_tame_main_elementary_auto}
\end{equation}
and ${a_i} \in \mathbb{F}_{q}[{X_1}, \ldots, \hat{X_{i}}, \ldots, {X_n}]$. 
If $q$ is odd or $q = 2^m$, $m \geq 2$ then we have ${\pi}_{q}\left(E_{a_{i}}^{(q)}\right) \in \mathop{\mathrm{Alt}}({\mathbb{F}_{q}^n})$. 
Namely, if $q$ is odd or $q = 2^m$, $m \geq 2$ then 
\begin{equation}
\mathop{\mathrm{sgn}}\left({\pi}_{q}\left(E_{a_{i}}^{(q)}\right)\right) = 1. 
\label{eqn:perm_tame_main_elementary_auto_sign_case1}
\end{equation}
If $q = 2$ then $\mathop{\mathrm{sgn}}\left({\pi}_{q}\left(E_{a_{i}}^{(q)}\right)\right)$ 
depends only on the number of monomials of the form 
$cX_1^{e_1}{\cdots}X_{i-1}^{e_{i-1}}X_{i+1}^{e_{i+1}}{\cdots}X_n^{e_n}$ 
with $c \in \mathbb{F}_{q}^{*}$ and ${e_1}, \ldots, {e_n} \geq 1$ appearing in the polynomial ${a_i}$. 
More precisely, if $q = 2$ then 
$\mathop{\mathrm{sgn}}\left({\pi}_{q}\left(E_{a_{i}}^{(q)}\right)\right) = \left(-1\right)^{M_{a_i}}$, 
where $M_{a_i}$ is the number of monomials of the form 
$cX_1^{e_1}{\cdots}X_{i-1}^{e_{i-1}}X_{i+1}^{e_{i+1}}{\cdots}X_n^{e_n}$ 
with $c \in \mathbb{F}_{q}^{*}$ and ${e_1}, \ldots, {e_n} \geq 1$ appearing in the polynomial ${a_i}$. 
\end{mainthm}
%===========================================================

%===========================================================
\begin{proof}
It suffices to assume that ${a_i} \in \mathbb{F}_{q}[{X_1}, \ldots, \hat{X_{i}}, \ldots, {X_n}]$ is a monomial 
(Similar discussion can be found in \cite[page 96, Proof of Theorem 5.2.1]{Mau01}). 
Let ${c_i} \in \mathbb{F}_{q}$, $N_{\hat{i}} := \{1, \ldots, i-1, i+1, \ldots, n\}$, and 
$\mathbf{e} := \left({e_1}, \ldots, \hat{e_i}, \ldots, {e_n}\right) \in \mathbb{Z}_{\geq 0}^{n-1}$ (${e_1}, \ldots, \hat{e_i}, \ldots, {e_n} \geq 0$). 
We put ${a_i} = {c_i}\prod_{\substack{1 \leq j \leq n \\ j \neq i}} {X_j^{e_j}}$ and 
\begin{equation*}
\begin{split}
E_{c_{i}, \mathbf{e}}^{(q)} & := 
\Biggl({X_1}, \ldots, {X_{i-1}}, {X_{i}} + {c_i}\prod_{\substack{1 \leq j \leq n \\ j \neq i}} {X_j^{e_j}}, {X_{i+1}}, \ldots, {X_n}\Biggr) \\
& = 
\Biggl({X_1}, \ldots, {X_{i-1}}, {X_{i}} + {c_i}\prod_{j \in N_{\hat{i}}} {X_j^{e_j}}, {X_{i+1}}, \ldots, {X_n}\Biggr) \in {\mathop{\mathrm{EA}}}_{n}\left(\mathbb{F}_{q}\right). %
\end{split}
\end{equation*}
In the following, we determine the value of $\mathop{\mathrm{sgn}}\left({\pi}_{q}\left(E_{c_{i}, \mathbf{e}}^{(q)}\right)\right)$ 
by decomposing the permutation ${\pi}_{q}\left(E_{c_{i}, \mathbf{e}}^{(q)}\right)$ 
as a product of transpositions. 
Let ${y_{1}}, \ldots, {y_{i-1}}, {y_{i+1}}, \ldots, {y_{n}}$ be elements of $\mathbb{F}_{q}$. 
We put $\mathbf{y} = \left({y_{1}}, \ldots, {y_{i-1}}, {y_{i+1}}, \ldots, {y_{n}}\right) \in \mathbb{F}_{q}^{n-1}$. 
%
\begin{comment}
\begin{equation*}
\mathbf{y} = \left({y_{1}}, \ldots, {y_{i-1}}, {y_{i+1}}, \ldots, {y_{n}}\right) \in \mathbb{F}_{q}^{n-1}. 
\end{equation*}
\end{comment}
%
For $\mathbf{y} \in \mathbb{F}_{q}^{n-1}$, 
we define the map $\lambda_{{c_i}, \mathbf{e}, \mathbf{y}}^{(q)}$ as follows: 
\begin{equation*}
\begin{split}
\begin{array}{ccccl}
\lambda_{{c_i}, \mathbf{e}, \mathbf{y}}^{(q)} : &       \mathbb{F}_{q}^{n}     & \longrightarrow &       \mathbb{F}_{q}^{n}                            & \\
                                       &      \rotatebox{90}{$\in$}   &                 &         \rotatebox{90}{$\in$}                       & \\
                                       &     ({x_1}, \ldots, {x_n})   & \longmapsto     & ({x_1}, \ldots, {x_n})                              & \text{if } {\exists}j \in N_{\hat{i}} \text{ s.t. } {x_j} \neq {y_j}, \\
                                       &     ({x_1}, \ldots, {x_n})   & \longmapsto     & {E_{c_{i}, \mathbf{e}}^{(q)}}({x_1}, \ldots, {x_n}) & {x_j} = {y_j} \text{ for all } j \in N_{\hat{i}}. %
\end{array}
\end{split}
\end{equation*}
It follows from the definition of the map $\lambda_{{c_i}, \mathbf{e}, \mathbf{y}}^{(q)}$ that 
\begin{equation*}
\lambda_{{c_i}, \mathbf{e}, \mathbf{y}}^{(q)}({x_1}, \ldots, {x_n}) = ({x_1}, \ldots, {x_n}) %
\end{equation*}
for each 
$({x_1}, \ldots, {x_n}) \in \mathbb{F}_{q}^{n} \setminus \{\left({y_{1}}, \ldots, {y_{i-1}}, y, {y_{i+1}}, \ldots, {y_{n}}\right) \mid y \in \mathbb{F}_{q} \}$ 
and 
\begin{equation*}
\lambda_{{c_i}, \mathbf{e}, \mathbf{y}}^{(q)}({x_1}, \ldots, {x_n}) = {E_{c_{i}, \mathbf{e}}^{(q)}}({x_1}, \ldots, {x_n}) %
\end{equation*}
for each 
$({x_1}, \ldots, {x_n}) \in \{\left({y_{1}}, \ldots, {y_{i-1}}, y, {y_{i+1}}, \ldots, {y_{n}}\right) \mid y \in \mathbb{F}_{q} \}$. 
Thus the map $\lambda_{{c_i}, \mathbf{e}, \mathbf{y}}^{(q)}$ is bijective 
and the inverse of $\lambda_{{c_i}, \mathbf{e}, \mathbf{y}}^{(q)}$ is $\lambda_{-{c_i}, \mathbf{e}, \mathbf{y}}^{(q)}$. 
Remark that 
$\lambda_{{c_i}, \mathbf{e}, \mathbf{y}}^{(q)}$ is the identity map 
if and only if $\mathbf{e} = \left(0, \ldots, 0\right)$ and ${c_i} = 0$, 
or there exists $j \in N_{\hat{i}}$ such that ${e_j} \neq 0$ and ${y_j} = 0$. 
Put 
\begin{equation*}
B\left(\lambda_{{c_i}, \mathbf{e}, \mathbf{y}}^{(q)}\right) := \left\{ ({x_1}, \ldots, {x_n}) \in \mathbb{F}_{q}^{n} \; \middle| \; \lambda_{{c_i}, \mathbf{e}, \mathbf{y}}^{(q)}({x_1}, \ldots, {x_n}) \neq ({x_1}, \ldots, {x_n}) \right\}. 
\end{equation*}
Since 
\begin{equation*}
B\left(\lambda_{{c_i}, \mathbf{e}, \mathbf{y}}^{(q)}\right) \cap B\left(\lambda_{{c_i}, \mathbf{e}, \mathbf{y}^{\prime}}^{(q)}\right) = {\emptyset} 
\end{equation*}
for any $\mathbf{y}^{\prime} = \left({y_{1}^{\prime}}, \ldots, {y_{i-1}^{\prime}}, {y_{i+1}^{\prime}}, \ldots, {y_{n}^{\prime}}\right) \in \mathbb{F}_{q}^{n-1}$ satisfying $\mathbf{y}^{\prime} \neq \mathbf{y}$, 
the permutation ${\pi}_{q}\left(E_{c_{i}, \mathbf{e}}^{(q)}\right)$ can be written as 
\begin{equation}
\begin{split}
{\pi}_{q}\left(E_{c_{i}, \mathbf{e}}^{(q)}\right) 
& = 
{\pi}_{q}\left(\prod_{{y_{1}}, \ldots, {y_{i-1}}, {y_{i+1}}, \ldots, {y_{n}} \in \mathbb{F}_{q}} \lambda_{{c_i}, \mathbf{e}, \mathbf{y}}^{(q)}\right) \\%
& = 
\prod_{{y_{1}}, \ldots, {y_{i-1}}, {y_{i+1}}, \ldots, {y_{n}} \in \mathbb{F}_{q}} {\pi}_{q}\left(\lambda_{{c_i}, \mathbf{e}, \mathbf{y}}^{(q)}\right), %
\label{eqn:perm_tame_main_elementary_monomial_auto_composite_perm}
\end{split}
\end{equation}
which is a composition of disjoint permutations on $\mathbb{F}_{q}^{n}$. 
We denote the number of distinct permutations other than the identity map 
in the right-hand side of Equation~\eqref{eqn:perm_tame_main_elementary_monomial_auto_composite_perm} by ${M_1}$. 
It is straightforward to check that ${M_1}$ satisfies the equation 
\begin{equation}
\displaystyle {M_1} = {\chi}\left({c_i}\right)^{\ell} \times {\prod}_{j \in N_{\hat{i}}} \left(q - {\delta}\left({e_j}\right)\right). 
\label{eqn:perm_tame_main_elementary_monomial_decomposition_m1}
\end{equation}

Next, we decompose each permutation ${\pi}_{q}\left(\lambda_{{c_i}, \mathbf{e}, \mathbf{y}}^{(q)}\right)$ 
as a composition of disjoint cycles on $\mathbb{F}_{q}^{n}$. 
In order to find such a decomposition, 
we define an equivalence relation ${\sim}$ on $\mathbb{F}_{q}$: 
$y \in \mathbb{F}_{q}$ and $y^{\prime} \in \mathbb{F}_{q}$ are equivalent 
if and only if there exists $l \in \{0, 1, \ldots, p-1 \}$ such that 
$y^{\prime} = y + l\left({c_i}\prod_{j \in N_{\hat{i}}} {y_j^{e_j}}\right)$. 
Note that the equivalence relation ${\sim}$ depends on the choice of ${y_{1}}, \ldots, {y_{i-1}}, {y_{i+1}}, \ldots, {y_{n}}$. 
Put 
\begin{equation*}
{C_y} := \{ y^{\prime} \in \mathbb{F}_{q} \mid y \sim y^{\prime} \}. 
\end{equation*}
We now choose a complete system of representatives $\mathcal{R}$ for the above equivalence relation ${\sim}$. 
Since $\mathcal{R}$ is a complete system of representatives, 
it follows that 
\begin{equation*}
{\sharp} \mathcal{R} = q/p = {p^m}/p = p^{m-1}. 
\end{equation*}
We write 
\begin{equation*}
\mathcal{R} := \{ {w_1}, \ldots, {w_{p^{m-1}}} \} \subset \mathbb{F}_{q}. 
\end{equation*}
For any ${w_s} \in \mathcal{R}$ ($1 \leq s \leq {p^{m-1}}$), we set 
${\mathbf{y}_{w_s}} := \left({y_{1}}, \ldots, {y_{i-1}}, {w_s}, {y_{i+1}}, \ldots, {y_{n}}\right)$. 
We define the bijective map $\lambda_{{c_i}, \mathbf{e}, {\mathbf{y}_{w_s}}}^{(q)}$ as follows: 
\begin{equation*}
\begin{split}
\begin{array}{ccccl}
\lambda_{{c_i}, \mathbf{e}, {\mathbf{y}_{w_s}}}^{(q)} : &       \mathbb{F}_{q}^{n}     & \longrightarrow &       \mathbb{F}_{q}^{n}           & \\
                                                        &      \rotatebox{90}{$\in$}   &                 &         \rotatebox{90}{$\in$}      & \\
                                                        &     ({x_1}, \ldots, {x_n})   & \longmapsto     & ({x_1}, \ldots, {x_n})             & \\
                                                        &                              &                 &                                    & \hspace*{-2.3cm} \text{if } {\exists}j \in N_{\hat{i}} \text{ s.t. } {x_j} \neq {y_j}, \text{ or } {x_i} \not\in C_{w_s}, \\
                                                        &     ({x_1}, \ldots, {x_n})   & \longmapsto     & {E_{c_{i}, \mathbf{e}}^{(q)}}({x_1}, \ldots, {x_n}) & \\
                                                        &                              &                 &                                    & \hspace*{-2.3cm} {x_j} = {y_j} \text{ for all } j \in N_{\hat{i}} \text{ and } {x_i} \in C_{w_s}. \\
\end{array}
\end{split}
\end{equation*}
Note that the map $\lambda_{{c_i}, \mathbf{e}, {\mathbf{y}_{w_s}}}^{(q)}$ is a cycle of length $p$, 
and the standard result from elementary group theory yields that its sign is $\left(-1\right)^{p-1}$, namely, 
\begin{equation}
\mathop{\mathrm{sgn}}\left({\pi}_{q}\left(\lambda_{{c_i}, \mathbf{e}, {\mathbf{y}_{w_s}}}^{(q)}\right)\right) = \left(-1\right)^{p-1}. %
\label{eqn:perm_tame_main_elementary_monomial_cycle_sign}
\end{equation}
We also remark that the map $\lambda_{{c_i}, \mathbf{e}, {\mathbf{y}_{w_s}}}^{(q)}$ is the identity map if and only if 
$\mathbf{e} = \left(0, \ldots, 0\right)$ and ${c_i} = 0$, 
or there exists $j \in N_{\hat{i}}$ such that ${e_j} \neq 0$ and ${y_j} = 0$. 
If $1 \leq {s_1}, {s_2} \leq {p^{m-1}}$ and ${s_1} \neq {s_2}$, 
then we have $C_{w_{s_1}} \cap C_{w_{s_2}} = \emptyset$. 
This yields a decomposition of the permutation ${\pi}_{q}\left(\lambda_{{c_i}, \mathbf{y}}\right)$ into a composition of disjoint cycles on $\mathbb{F}_{q}^{n}$: 
\begin{equation}
{\pi}_{q}\left(\lambda_{{c_i}, \mathbf{e}, \mathbf{y}}^{(q)}\right) = %
{\pi}_{q}\left(\prod_{s = 1}^{{p^{m-1}}} \lambda_{{c_i}, \mathbf{e}, {\mathbf{y}_{w_s}}}^{(q)}\right) = %
\prod_{s = 1}^{{p^{m-1}}} {\pi}_{q}\left(\lambda_{{c_i}, \mathbf{e}, {\mathbf{y}_{w_s}}}^{(q)}\right). %
\label{eqn:perm_tame_main_elementary_monomial_auto_composite_cyclic}
\end{equation}
We denote the number of disjoint cycles other than the identify map in Equation~\eqref{eqn:perm_tame_main_elementary_monomial_auto_composite_cyclic} by ${M_2}$. 
By counting the number of disjoint cycles appearing in Equation~\eqref{eqn:perm_tame_main_elementary_monomial_auto_composite_cyclic}, 
we have 
\begin{equation}
{M_2} = {\chi}\left({c_i}\right)^{\ell} \times p^{m-1}. %
\label{eqn:perm_tame_main_elementary_monomial_decomposition_m2}
\end{equation}
Now we determine the sign of the permutation ${\pi}_{q}\left(E_{c_{i}, \mathbf{e}}^{(q)}\right)$. 
By Equation~\eqref{eqn:perm_tame_main_elementary_monomial_auto_composite_perm} 
through Equation~\eqref{eqn:perm_tame_main_elementary_monomial_decomposition_m2}, 
we obtain 
\begin{equation*}
\mathop{\mathrm{sgn}}\left({\pi}_{q}\left(E_{c_{i}, \mathbf{e}}^{(q)}\right)\right) = \left(-1\right)^{M}, 
\end{equation*}
where 
\begin{equation}
\begin{split}
M %
& = %
{M_1} \times {M_2} \times \left(p-1\right) \\%
& = %
{\chi}\left({c_i}\right)^{\ell} \times p^{m-1} \times \left(p-1\right) \times \left({\prod}_{j \in N_{\hat{i}}} \left(q - {\delta}\left({e_j}\right)\right)\right). %
\end{split}
\label{eqn:perm_tame_main_elementary_monomial_decomposition_m}
\end{equation}
If $q$ is odd (and thus $p$ is odd) or $q = 2^m$, $m \geq 2$, we have 
$M \equiv 0 \bmod 2$ and if $q = 2$ (namely, $p = 2$ and $m = 1$), 
we have $M \equiv {\chi}\left({c_i}\right)^{\ell} \times {\prod}_{j \in N_{\hat{i}}} {\delta}\left({e_j}\right) \bmod 2$. 
Therefore if $q$ is odd or $q = 2^m$, $m \geq 2$ then 
${\pi}_{q}\left(E_{c_{i}, \mathbf{e}}^{(q)}\right) \in \mathop{\mathrm{Alt}}({\mathbb{F}_{q}^n})$. 
Hence, we have ${\pi}_{q}\left(E_{a_{i}}^{(q)}\right) \in \mathop{\mathrm{Alt}}({\mathbb{F}_{q}^n})$, 
namely, $\mathop{\mathrm{sgn}}\left({\pi}_{q}\left(E_{a_{i}}^{(q)}\right)\right) = 1$. 
On the other hand, if $q = 2$ then $\mathop{\mathrm{sgn}}\left({\pi}_{q}\left(E_{c_{i}, \mathbf{e}}^{(q)}\right)\right) = \left(-1\right)^{{\chi}\left({c_i}\right)^{\ell} \times {\prod}_{j \in N_{\hat{i}}} {\delta}\left({e_j}\right)}$. 
Thus, $\mathop{\mathrm{sgn}}\left({\pi}_{q}\left(E_{c_{i}, \mathbf{e}}^{(q)}\right)\right) = -1$ 
if and only if ${c_i} \neq 0$ and ${\prod}_{j \in N_{\hat{i}}} {\delta}\left({e_j}\right) = 1$, 
or more generally, $\mathop{\mathrm{sgn}}\left({\pi}_{q}\left(E_{a_{i}}^{(q)}\right)\right) = -1$ 
if and only if the number of monomials of the form 
$cX_1^{e_1}{\cdots}X_{i-1}^{e_{i-1}}X_{i+1}^{e_{i+1}}{\cdots}X_n^{e_n}$ 
with $c \in \mathbb{F}_{q}^{*}$ and ${e_1}, \ldots, {e_n} \geq 1$ appearing in the polynomial ${a_i}$, is odd. 
This completes the proof. 
\end{proof}
%===========================================================

%===========================================================
\begin{cor}
\label{cor:perm_tame_main_elementary_auto_sign}
If $q$ is odd or $q = 2^m$, $m \geq 2$ then 
we have ${\pi}_{q}\left({\mathop{\mathrm{EA}}}_{n}\left(\mathbb{F}_{q}\right)\right) \subset \mathop{\mathrm{Alt}}({\mathbb{F}_{q}^n})$. 
\end{cor}
%===========================================================

%===========================================================
\begin{rem}
\label{rem:perm_tame_main_elementary_auto_sign_multiplicative_character_replace}
{\rm 
Let ${\theta}: \mathbb{F}_{q} \rightarrow \mathbb{C}$ be a map satisfying that 
${\theta}\left(c\right) = 0$ when $c = 0$, 
and ${\theta}\left(c\right) = 1$ when $c \in \mathbb{F}_{q}^{*}$. 
Then one can also prove Main Theorem~\ref{thm:perm_tame_main_elementary_auto_sign} 
by replacing ${\chi}\left({c_i}\right)^{\ell}$ with ${\theta}\left({c_i}\right)$ 
in the proof of Main Theorem~\ref{thm:perm_tame_main_elementary_auto_sign}. 
}
\end{rem}
%===========================================================

%===========================================================
\begin{rem}
\label{rem:perm_tame_main_elementary_auto_sign_our_result_is_more_general}
{\rm 
Note that Main Theorem~\ref{thm:perm_tame_main_elementary_auto_sign} is similar to \cite[Lemma~6.4]{MR15}, 
but Main Theorem~\ref{thm:perm_tame_main_elementary_auto_sign} is more general result than \cite[Lemma~6.4]{MR15}. 
}
\end{rem}
%===========================================================

%===========================================================
\begin{ex}
\label{ex:perm_tame_main_elementary_auto_sign_example}
{\rm 
Let ${a_i} := {\alpha}{X_1}\cdots{X_{i-1}}{X_{i+1}}\cdots{X_n}, {b_i} := {\beta}{X_1^{2}}{X_2}\cdots{X_{i-1}}{X_{i+1}}\cdots{X_n} \in \mathbb{F}_{q}[{X_1}, \ldots, \hat{X_{i}}, \ldots, {X_n}]$, 
and ${\alpha}, {\beta} \in \mathbb{F}_{q}^{*}$. 
We consider the sign of the permutations induced by the following elementary automorphisms: 
\begin{equation*}
E_{a_{i}}^{(q)} = ({X_1}, \ldots, {X_{i-1}}, {X_{i}} + {a_i}, {X_{i+1}}, \ldots, {X_n}) \in {\mathop{\mathrm{EA}}}_{n}\left(\mathbb{F}_{q}\right), %
\end{equation*}
\begin{equation*}
E_{b_{i}}^{(q)} = ({X_1}, \ldots, {X_{i-1}}, {X_{i}} + {b_i}, {X_{i+1}}, \ldots, {X_n}) \in {\mathop{\mathrm{EA}}}_{n}\left(\mathbb{F}_{q}\right), %
\end{equation*}
and 
\begin{equation*}
E_{a_{i} + b_{i}}^{(q)} = ({X_1}, \ldots, {X_{i-1}}, {X_{i}} + {a_i} + {b_i}, {X_{i+1}}, \ldots, {X_n}) \in {\mathop{\mathrm{EA}}}_{n}\left(\mathbb{F}_{q}\right). %
\end{equation*}
We first assume that $q$ is odd or $q = 2^m$, $m \geq 2$. 
Then by Main Theorem~\ref{thm:perm_tame_main_elementary_auto_sign}, 
we have $\mathop{\mathrm{sgn}}\left({\pi}_{q}\left(E_{a_{i}}^{(q)}\right)\right) = \mathop{\mathrm{sgn}}\left({\pi}_{q}\left(E_{b_{i}}^{(q)}\right)\right) = 1$. 
By the fact that ${\pi}_{q}$ and $\mathop{\mathrm{sgn}}$ are group homomorphisms, 
we also have $\mathop{\mathrm{sgn}}\left({\pi}_{q}\left(E_{a_{i} + b_{i}}^{(q)}\right)\right) = \mathop{\mathrm{sgn}}\left({\pi}_{q}\left(E_{a_{i}}^{(q)}\right)\right) \times \mathop{\mathrm{sgn}}\left({\pi}_{q}\left(E_{b_{i}}^{(q)}\right)\right) = 1$. 
We next suppose that $q = 2$. 
Since ${\alpha}, {\beta} \in \mathbb{F}_{q}^{*}$, we remark that ${\alpha} = {\beta} = 1$. 
From Main Theorem~\ref{thm:perm_tame_main_elementary_auto_sign}, 
we have $\mathop{\mathrm{sgn}}\left({\pi}_{q}\left(E_{a_{i}}^{(q)}\right)\right) = \mathop{\mathrm{sgn}}\left({\pi}_{q}\left(E_{b_{i}}^{(q)}\right)\right) = -1$. 
Again, by the fact that ${\pi}_{q}$ and $\mathop{\mathrm{sgn}}$ are group homomorphisms, 
we obtain $\mathop{\mathrm{sgn}}\left({\pi}_{q}\left(E_{a_{i} + b_{i}}^{(q)}\right)\right) = \mathop{\mathrm{sgn}}\left({\pi}_{q}\left(E_{a_{i}}^{(q)}\right)\right) \times \mathop{\mathrm{sgn}}\left({\pi}_{q}\left(E_{b_{i}}^{(q)}\right)\right) = 1$. 
We can directly prove that $\mathop{\mathrm{sgn}}\left({\pi}_{q}\left(E_{a_{i} + b_{i}}^{(q)}\right)\right) = 1$ 
by using the fact that ${\pi}_{q}\left(E_{a_{i} + b_{i}}^{(q)}\right) = {\pi}_{q}\left(\left({X_1}, \ldots, {X_n}\right)\right)$. 
}
\end{ex}
%===========================================================

%EOF

%%%%%%%%%%%%%%%%%%%%%%%%%%%%%%%%%%%%%%%%%%%%%%%%%%
\section{Sign of permutations induced by affine automorphisms}
\label{sec:perm_tame_main_affine}
%%%%%%%%%%%%%%%%%%%%%%%%%%%%%%%%%%%%%%%%%%%%%%%%%%

In this section, we consider the sign of permutations induced by affine automorphisms over finite fields. 
Suppose that $\widetilde{A_\mathbf{b}^{(q)}}$ 
is an affine automorphism over a finite field, where, 
\begin{equation}
\widetilde{A_\mathbf{b}^{(q)}} := \left(\left(\sum_{i=1}^{n} {a_{1, i}}{X_i}\right) + {b_1}, \ldots, \left(\sum_{i=1}^{n} {a_{n, i}}{X_i}\right) + {b_n}\right) \in {\mathop{\mathrm{Aff}}}_{n}\left(\mathbb{F}_{q}\right), %
\label{eqn:perm_tame_main_aff_auto}
\end{equation}
and $a_{i, j}, {b_i} \in \mathbb{F}_{q}$ for $1 \leq i, j \leq n$. 
We also assume that $A^{(q)}$ is the homogeneous part (linear automorphism) of the affine automorphism~\eqref{eqn:perm_tame_main_aff_auto}, 
namely, 
\begin{equation}
A^{(q)} = \left(\sum_{i=1}^{n} {a_{1, i}}{X_i}, \ldots, \sum_{i=1}^{n} {a_{n, i}}{X_i}\right) \in {\mathop{\mathrm{Aff}}}_{n}\left(\mathbb{F}_{q}\right). %
\label{eqn:perm_tame_main_lin_auto}
\end{equation}
By Equation~\eqref{eqn:perm_tame_main_aff_auto_semidirect_product} 
and by Main Theorem~\ref{thm:perm_tame_main_elementary_auto_sign}, 
we obtain 
\begin{equation}
\mathop{\mathrm{sgn}}\left({\pi}_{q}\left(\widetilde{A_\mathbf{b}^{(q)}}\right)\right) %
= \mathop{\mathrm{sgn}}\left({\pi}_{q}\left(A^{(q)}\right)\right). %
\label{eqn:perm_tame_main_aff_auto_sign_equal_lin_auto_sign}
\end{equation}
Thus, it is sufficient to consider the sign of affine automorphisms over finite field of the form~\eqref{eqn:perm_tame_main_lin_auto}. 
We put 
\begin{equation}
T_{i, j} := \left({X_1}, \ldots, {X_{i-1}}, {X_j}, {X_{i+1}}, \ldots, {X_{j-1}}, {X_i}, {X_{j+1}}, \ldots, {X_n}\right) \in {\mathop{\mathrm{Aff}}}_{n}\left(\mathbb{F}_{q}\right), 
\label{eqn:perm_tame_main_aff_elementary_mat_t_lin_auto}
\end{equation}
\begin{equation}
{D_{i}}\left(c\right) := \left({X_1}, \ldots, {X_{i-1}}, c{X_i}, {X_{i+1}}, \ldots, {X_n}\right) \in {\mathop{\mathrm{Aff}}}_{n}\left(\mathbb{F}_{q}\right), 
\label{eqn:perm_tame_main_aff_elementary_mat_d_lin_auto}
\end{equation}
and 
\begin{equation}
{R_{i, j}}\left(c\right) := \left({X_1}, \ldots, {X_{i-1}}, {X_i} + c{X_j}, {X_{i+1}}, \ldots, {X_n}\right) \in {\mathop{\mathrm{Aff}}}_{n}\left(\mathbb{F}_{q}\right), 
\label{eqn:perm_tame_main_aff_elementary_mat_r_lin_auto}
\end{equation}
where $c \in \mathbb{F}_{q}^{*}$. 
It is easy to see that 
\begin{equation*}
{T_{i, j}^{2}} := {T_{i, j}} \circ {T_{i, j}} = \left({X_1}, \ldots, {X_{n}}\right), 
\end{equation*}
\begin{equation*}
{D_{i}}\left(c\right)^{q - 1} := \underbrace{{D_{i}}\left(c\right) \circ \cdots \circ {D_{i}}\left(c\right)}_{\text{$q-1$ times}} %
= \left({X_1}, \ldots, {X_{n}}\right), 
\end{equation*}
and 
\begin{equation*}
{R_{i, j}}\left(c\right)^{p} := \underbrace{{R_{i, j}}\left(c\right) \circ \cdots \circ {R_{i, j}}\left(c\right)}_{\text{$p$ times}} %
= \left({X_1}, \ldots, {X_{n}}\right). 
\end{equation*}
Since each invertible matrix is a product of elementary matrices (\cite[Proposition~2.18]{Art91}), 
a linear automorphism can be written as a finite composition of linear automorphisms 
of form~\eqref{eqn:perm_tame_main_aff_elementary_mat_t_lin_auto}, \eqref{eqn:perm_tame_main_aff_elementary_mat_d_lin_auto}, 
and \eqref{eqn:perm_tame_main_aff_elementary_mat_r_lin_auto}. 
Namely, there exists $\ell_{A} \in \mathbb{Z}_{\geq 0}$ 
and linear automorphisms ${M_{1}^{(A)}}, \ldots, {M_{\ell_{A}}^{(A)}}$ such that 
\begin{equation}
A^{\left(q\right)} = {M_{1}^{\left(A\right)}} \circ \cdots \circ {M_{\ell_{A}}^{\left(A\right)}}, 
\label{eqn:perm_tame_main_aff_auto_elementary_matrices_decomposition}
\end{equation}
where ${M_{i}^{(A)}}$ is a linear automorphism of form~\eqref{eqn:perm_tame_main_aff_elementary_mat_t_lin_auto}, 
\eqref{eqn:perm_tame_main_aff_elementary_mat_d_lin_auto}, or \eqref{eqn:perm_tame_main_aff_elementary_mat_r_lin_auto} 
for each $i$ ($1 \leq i \leq {\ell_{A}}$). 
We remark that the representation~\eqref{eqn:perm_tame_main_aff_auto_elementary_matrices_decomposition} is not unique in general. 
Since ${\pi}_{q}$ and $\mathop{\mathrm{sgn}}$ are group homomorphisms, 
we obtain 
\begin{equation}
\mathop{\mathrm{sgn}}\left({\pi}_{q}\left(A^{(q)}\right)\right) = 
\prod_{i = 1}^{\ell_{A}} \mathop{\mathrm{sgn}}\left({\pi}_{q}\left({M_{i}^{(A)}}\right)\right). %
\label{eqn:perm_tame_main_aff_auto_elementary_matrices_decomposition_sign}
\end{equation}
Therefore, it is sufficient to consider 
the sign of the permutation induced by a linear automorphism 
of form~\eqref{eqn:perm_tame_main_aff_elementary_mat_t_lin_auto}, 
\eqref{eqn:perm_tame_main_aff_elementary_mat_d_lin_auto}, and \eqref{eqn:perm_tame_main_aff_elementary_mat_r_lin_auto}. 
We use for the symbols ${N_T}\left(A\right)$, ${N_D}\left(A\right)$, and ${N_R}\left(A\right)$ 
to represent the number of linear automorphisms of form~\eqref{eqn:perm_tame_main_aff_elementary_mat_t_lin_auto}, 
\eqref{eqn:perm_tame_main_aff_elementary_mat_d_lin_auto}, and \eqref{eqn:perm_tame_main_aff_elementary_mat_r_lin_auto} 
appearing in \eqref{eqn:perm_tame_main_aff_auto_elementary_matrices_decomposition}, respectively. 
Then we have 
\begin{equation*}
{N_T}\left(A\right) + {N_D}\left(A\right) + {N_R}\left(A\right) = {\ell_A}. 
\end{equation*}
Furthermore, we suppose that $\left\{ i_{1}^{(A)}, \ldots, i_{{N_D}(A)}^{(A)} \right\}$ 
is the subset of $\left\{1, \ldots, {\ell_A}\right\}$ satisfying the following two conditions: 

%===========================================================
\medskip
\noindent
%%%%%%%%%%%%%%%%%%%%
{(i): } 
%%%%%%%%%%%%%%%%%%%%
There exist ${u_j^{(A)}} \in \{1, \ldots, n\}$ and $c_{j}^{(A)} \in \mathbb{F}_{q}^{*}$ such that %
${M_{i_j^{(A)}}^{(A)}} = D_{u_j^{(A)}}\left(c_{j}^{(A)}\right)$ for $1 \leq j \leq {N_D}(A)$. 

\medskip
\noindent
%%%%%%%%%%%%%%%%%%%%
{(ii): } 
%%%%%%%%%%%%%%%%%%%%
For $i \in \left\{1, \ldots, {\ell_A}\right\} \setminus \{ i_{1}^{(A)}, \ldots, i_{{N_D}(A)}^{(A)} \}$, 
${M_{i}^{(A)}}$ is of form~\eqref{eqn:perm_tame_main_aff_elementary_mat_t_lin_auto} or \eqref{eqn:perm_tame_main_aff_elementary_mat_r_lin_auto}. 

\medskip

In the following, we determine the sign of permutations induced 
by the above three types of linear automorphisms of form~\eqref{eqn:perm_tame_main_aff_elementary_mat_t_lin_auto}, 
\eqref{eqn:perm_tame_main_aff_elementary_mat_d_lin_auto}, and \eqref{eqn:perm_tame_main_aff_elementary_mat_r_lin_auto}. 

%===========================================================
\begin{lem}
\label{thm:perm_tame_main_affine_auto_sign_elementary_mat_t}
{\normalfont\bfseries (Sign of ${\pi_{q}}\left({T_{i, j}}\right)$)}\ 
Suppose that ${T_{i, j}} \in {\mathop{\mathrm{Aff}}}_{n}\left(\mathbb{F}_{q}\right)$ 
is a linear automorphism of form~\eqref{eqn:perm_tame_main_aff_elementary_mat_t_lin_auto}. 
If $n \geq 2$, then 
\begin{equation}
\mathop{\mathrm{sgn}}\left({\pi_{q}}\left({T_{i, j}}\right)\right) = %
\begin{cases}
1 & \left(q = {2^m}, m \geq 2 \text{ or } q = 2, n \geq 3 \right), \\%
-1 & \left(q = 2 \text{ and } n = 2 \right), \\%
\left(-1\right)^{\frac{q-1}{2}} & \left(q \text{ is odd}\right). %
\end{cases}
\label{eqn:perm_tame_main_aff_elementary_mat_t_sign}
\end{equation}
\end{lem}
%===========================================================

%===========================================================
\begin{proof}
Put 
\begin{align*}
\mathcal{B}\left({T_{i, j}}\right) 
& := %
\{ \left({x_1}, \ldots, {x_n}\right) \in \mathbb{F}_{q}^{n} \; \vert \; %
{T_{i, j}}\left(\left({x_1}, \ldots, {x_n}\right)\right) \neq \left({x_1}, \ldots, {x_n}\right) \} \\%
& = %
\{ \left({x_1}, \ldots, {x_n}\right) \in \mathbb{F}_{q}^{n} \; \vert \; {x_i} \neq {x_j} \}. %
\end{align*}
Since 
\begin{equation*}
\mathcal{B}\left({T_{i, j}}\right) = \mathbb{F}_{q}^{n} \setminus \{ \left({x_1}, \ldots, {x_n}\right) \in \mathbb{F}_{q}^{n} \; \vert \; {x_i} = {x_j} \}, %
\end{equation*}
we have $\sharp \mathcal{B}\left({T_{i, j}}\right) = q^{n} - q^{n-2} \times q = q^{n-1}\left(q-1\right)$. 
Hence one can see that the permutation ${\pi_{q}}\left({T_{i, j}}\right)$ is a compositon of $q^{n-1}\left(q-1\right)/2$ transpositions on $\mathbb{F}_{q}^{n}$. 
We write ${N_{T_{i, j}}} = q^{n-1}\left(q-1\right)/2$. 

\medskip
\noindent
%%%%%%%%%%%%%%%%%%%%
{\bf Case~1. } $q = {2^m}, m \geq 2$. \\
%%%%%%%%%%%%%%%%%%%%
Since $q^{n-1}/2 = 2^{m\left(n-1\right)-1} \geq 2$, we have ${N_{T_{i, j}}} \equiv 0 \bmod 2$. 
Therefore, ${\pi_{q}}\left({T_{i, j}}\right)$ is an even permutation. 

\medskip
\noindent
%%%%%%%%%%%%%%%%%%%%
{\bf Case~2. } $q = 2$. \\
%%%%%%%%%%%%%%%%%%%%
In this case, we see that ${N_{T_{i, j}}} = 2^{n-2}$. 
If $n \geq 3$, we obtain ${N_{T_{i, j}}} \equiv 0 \bmod 2$. 
Otherwise ${N_{T_{i, j}}} = 1$. 

\medskip
\noindent
%%%%%%%%%%%%%%%%%%%%
{\bf Case~3. } $q$ is odd. \\
%%%%%%%%%%%%%%%%%%%%
If $q$ is odd then $q \equiv 1 \bmod 4$ or $q \equiv 3 \bmod 4$. 
From this fact, we obtain 
\begin{equation*}
{N_{T_{i, j}}} \equiv \frac{q - 1}{2} \equiv %
\begin{cases}
0 & \left(q \equiv 1 \bmod 4 \right), \\%
1 & \left(q \equiv 3 \bmod 4 \right). %
\end{cases}
\end{equation*}
Thus, Equation~\eqref{eqn:perm_tame_main_aff_elementary_mat_t_sign} holds. 
\end{proof}
%===========================================================

%===========================================================
\begin{lem}
\label{thm:perm_tame_main_affine_auto_sign_elementary_mat_d}
{\normalfont\bfseries (Sign of ${\pi_{q}}\left({D_{i}}\left(c\right)\right)$)}\ 
Suppose that ${D_{i}}\left(c\right) \in {\mathop{\mathrm{Aff}}}_{n}\left(\mathbb{F}_{q}\right)$ 
is a linear automorphism of form~\eqref{eqn:perm_tame_main_aff_elementary_mat_d_lin_auto}. 
If $n \geq 2$, then 
\begin{equation}
\mathop{\mathrm{sgn}}\left({\pi_{q}}\left({D_{i}}\left(c\right)\right)\right) = %
\begin{cases}
1 & \left(q \text{ is even}\right), \\%
\left(-1\right)^{\rm{ord}_{\mathbb{F}_{q}^{*}}\left(c\right)} & \left(q \text{ is odd}\right). %
\end{cases}
\label{eqn:perm_tame_main_aff_elementary_mat_d_sign}
\end{equation}
\end{lem}
%===========================================================

%===========================================================
\begin{proof}
Let us suppose that $q$ is even. 
We assume that $\mathop{\mathrm{sgn}}\left({\pi_{q}}\left({D_{i}}\left(c\right)\right)\right) = -1$. 
Since $\mathop{\mathrm{sgn}}$ and ${\pi_{q}}$ are group homomorphisms, 
we have 
\begin{equation*}
\mathop{\mathrm{sgn}}\left({\pi_{q}}\left({D_{i}}\left(c\right)^{q-1}\right)\right) = %
\mathop{\mathrm{sgn}}\left({\pi_{q}}\left({D_{i}}\left(c\right)\right)\right)^{q-1} = \left(-1\right)^{q-1} = -1. 
\end{equation*}
On the other hand, from ${D_{i}}\left(c\right)^{q - 1} = \left({X_1}, \ldots, {X_{n}}\right)$, 
we must have 
\begin{equation*}
\mathop{\mathrm{sgn}}\left({\pi_{q}}\left({D_{i}}\left(c\right)^{q-1}\right)\right) = %
\mathop{\mathrm{sgn}}\left({\pi_{q}}\left(\left({X_1}, \ldots, {X_{n}}\right)\right)\right) = 1. 
\end{equation*}
This is a contradiction. 
Therefore, 
$\mathop{\mathrm{sgn}}\left({\pi_{q}}\left({D_{i}}\left(c\right)\right)\right) = 1$. 

Next suppose that $q$ is odd. 
We define the map $H_{c}: \mathbb{F}_{q} \rightarrow \mathbb{F}_{q}$ as follows: 
\begin{equation*}
\begin{split}
\begin{array}{cccc}
H_{c} : &     \mathbb{F}_{q}    & \longrightarrow &     \mathbb{F}_{q}  \\
        & \rotatebox{90}{$\in$} &                 & \rotatebox{90}{$\in$} \\
        &            x          & \longmapsto     &            cx.        %
\end{array}
\end{split}
\end{equation*}
The map $H_{c}$ is bijective, and the inverse map is $H_{c^{-1}}$. 
Since we can regard the map $H_{c}$ as a permutation on $\mathbb{F}_{q}$, 
it is obious that 
\begin{equation*}
\mathop{\mathrm{sgn}}\left({\pi_{q}}\left({D_{i}}\left(c\right)\right)\right) = \mathop{\mathrm{sgn}}\left(H_{c}\right). 
\end{equation*}
Let $g$ be a generator of the multiplicative group $\mathbb{F}_{q}^{*}$. 
We put $c = g^h$, $0 \leq h = \text{ord}_{\mathbb{F}_{q}^{*}}\left(c\right) \leq q - 2$. 
If $h = 0$ then $\mathop{\mathrm{sgn}}\left(H_{c}\right) = 1$. 
We assume that $h \neq 0$. 
If $h = 1$ then the map $H_{c}$ is the length $q-1$ cycle 
$\left(g^{0} \; g^{1} \; \cdots \; g^{q-2} \right)$ as a permutation on $\mathbb{F}_{q}$, 
and hence $H_{c} = \left(g^{q-3} \; g^{q-2}\right) \circ \left(g^{q-4} \; g^{q-3}\right) \circ \cdots \circ \left(g^{1} \; g^{2}\right) \circ \left(g^{0} \; g^{1}\right)$, 
the product of $q-2$ transpositions as a permutation on $\mathbb{F}_{q}$. 
This yields that for $1 \leq h \leq q - 2$, the map $H_{c}$ is 
the product of $h$ copies of the length $q-1$ cycle $\left(g^{0} \; g^{1} \; \cdots \; g^{q-2} \right)$, 
namely, the product of $h \times \left(q - 2\right)$ transpositions as a permutation on $\mathbb{F}_{q}$. 
Therefore, we obtain $\mathop{\mathrm{sgn}}\left(H_{c}\right) = \left(-1\right)^{h\left(q - 2\right)}$. 
Since $q$ is odd, it satisfies that $\left(-1\right)^{h\left(q - 2\right)} = \left(\left(-1\right)^{q-2}\right)^{h} = \left(-1\right)^{h}$. 
Hence 
\begin{equation}
\mathop{\mathrm{sgn}}\left(H_{c}\right) = \left(-1\right)^{h} %
= \left(-1\right)^{\text{ord}_{\mathbb{F}_{q}^{*}}\left(c\right)} 
\label{eqn:perm_tame_main_aff_permutation_hc_sign}
\end{equation}
for $c \in \mathbb{F}_{q}^{*}$, $c \neq 1$. 
Equation~\eqref{eqn:perm_tame_main_aff_permutation_hc_sign} is obviously true for $c = 1$. 
Thus the assertion holds. 
\end{proof}
%===========================================================

%===========================================================
\begin{lem}
\label{thm:perm_tame_main_affine_auto_sign_elementary_mat_r}
{\normalfont\bfseries (Sign of ${\pi_{q}}\left({R_{i, j}}\left(c\right)\right)$)}\ 
Suppose that ${R_{i, j}}\left(c\right) \in {\mathop{\mathrm{Aff}}}_{n}\left(\mathbb{F}_{q}\right)$ 
is a linear automorphism of form~\eqref{eqn:perm_tame_main_aff_elementary_mat_r_lin_auto}. 
If $n \geq 2$, then 
\begin{equation}
\mathop{\mathrm{sgn}}\left({\pi_{q}}\left({R_{i, j}}\left(c\right)\right)\right) = 1. %
\label{eqn:perm_tame_main_aff_elementary_mat_r_sign}
\end{equation}
\end{lem}
%===========================================================

%===========================================================
\begin{proof}
By ${R_{i, j}}\left(c\right) %
\in {\mathop{\mathrm{Aff}}}_{n}\left(\mathbb{F}_{q}\right) \cap {\mathop{\mathrm{EA}}}_{n}\left(\mathbb{F}_{q}\right) %
\subset {\mathop{\mathrm{EA}}}_{n}\left(\mathbb{F}_{q}\right)$, 
it follows immediately from Main Theorem~\ref{thm:perm_tame_main_elementary_auto_sign}. 
\end{proof}
%===========================================================

\medskip

We are now in the position to prove the main result of this section. 

%===========================================================
\begin{mainthm}
\label{thm:perm_tame_main_affine_auto_sign}
{\rm\bfseries (Sign of affine automorphisms)} \ 
With notation as above, the following assertions are hold: 

\noindent
{\normalfont (1)}\ 
If $q = {2^m}$, $m \geq 2$ then 
\begin{equation}
\mathop{\mathrm{sgn}}\left({\pi}_{q}\left(\widetilde{A_\mathbf{b}^{(q)}}\right)\right) = 1. 
\label{eqn:perm_tame_main_aff_auto_sign_case_q_even_not_two}
\end{equation}

\noindent
{\normalfont (2)}\ If $q = 2$ and $n = 2$ then 
\begin{equation}
\mathop{\mathrm{sgn}}\left({\pi}_{q}\left(\widetilde{A_\mathbf{b}^{(q)}}\right)\right) = \left(-1\right)^{{N_T}\left(A\right)}. 
\label{eqn:perm_tame_main_aff_auto_sign_case_q_two_n_two}
\end{equation}

\noindent
{\normalfont (3)}\ If $q = 2$ and $n \geq 3$ then 
\begin{equation}
\mathop{\mathrm{sgn}}\left({\pi}_{q}\left(\widetilde{A_\mathbf{b}^{(q)}}\right)\right) = 1. 
\label{eqn:perm_tame_main_aff_auto_sign_case_q_two_n_geq_3}
\end{equation}

\noindent
{\normalfont (4)}\ 
If $q$ is odd then 
\begin{equation}
\mathop{\mathrm{sgn}}\left({\pi}_{q}\left(\widetilde{A_\mathbf{b}^{(q)}}\right)\right) %
= \left(-1\right)^{\left(\sum_{j = 1}^{{N_D}(A)} {\rm ord}_{\mathbb{F}_{q}^{*}}\left({c_j^{(A)}}\right) \right) + \frac{q-1}{2}{N_T}(A)}. 
\label{eqn:perm_tame_main_aff_auto_sign_case_q_odd_general}
\end{equation}
In particular, if $q \equiv 1 \bmod 4$ then 
\begin{equation}
\mathop{\mathrm{sgn}}\left({\pi}_{q}\left(\widetilde{A_\mathbf{b}^{(q)}}\right)\right) %
= \left(-1\right)^{\left(\sum_{j = 1}^{{N_D}(A)} {\rm ord}_{\mathbb{F}_{q}^{*}}\left({c_j^{(A)}}\right) \right)}. 
\label{eqn:perm_tame_main_aff_auto_sign_case_q_odd_1_mod_4}
\end{equation}
\end{mainthm}
%===========================================================

%===========================================================
\begin{proof}
The assertions (1) through (4) follow immediately from 
Equation~\eqref{eqn:perm_tame_main_aff_auto_sign_equal_lin_auto_sign}, 
Equation~\eqref{eqn:perm_tame_main_aff_auto_elementary_matrices_decomposition}, 
and Lemma~\ref{thm:perm_tame_main_affine_auto_sign_elementary_mat_t} through 
Lemma~\ref{thm:perm_tame_main_affine_auto_sign_elementary_mat_r}. 
This completes the proof. 
\end{proof}
%===========================================================

%===========================================================
\begin{cor}
\label{cor:perm_tame_main_affine_auto_sign}
If $q = 2^m$ and $m \geq 2$, or $q = 2$ and $n \geq 3$ then 
we have ${\pi}_{q}\left({\mathop{\mathrm{Aff}}}_{n}\left(\mathbb{F}_{q}\right)\right) \subset \mathop{\mathrm{Alt}}({\mathbb{F}_{q}^n})$. 
\end{cor}
%===========================================================

%===========================================================
\begin{rem}
\label{rem:perm_tame_main_affine_auto_sign_uniqueness}
{\rm 
One can see that 
Equation~\eqref{eqn:perm_tame_main_aff_auto_sign_case_q_two_n_two}, 
\eqref{eqn:perm_tame_main_aff_auto_sign_case_q_two_n_geq_3}, 
\eqref{eqn:perm_tame_main_aff_auto_sign_case_q_odd_general}, 
and \eqref{eqn:perm_tame_main_aff_auto_sign_case_q_odd_1_mod_4} 
depend on the representation~\eqref{eqn:perm_tame_main_aff_auto_elementary_matrices_decomposition} which is not unique. 
However, since ${\pi}_{q}\left(\widetilde{A_\mathbf{b}^{(q)}}\right)$ is uniquely determined as a permutation on $\mathbb{F}_{q}^{n}$, 
$\mathop{\mathrm{sgn}}\left({\pi}_{q}\left(\widetilde{A_\mathbf{b}^{(q)}}\right)\right)$ does not depend on 
the representation~\eqref{eqn:perm_tame_main_aff_auto_elementary_matrices_decomposition}. 
}
\end{rem}
%===========================================================

%===========================================================
\begin{ex}
\label{ex:perm_tame_main_affine_auto_sign_example}
{\rm 
Let ${\alpha}, {\beta} \in \mathbb{F}_{q}^{*}$. 
We consider the sign of the permutation induced by the affine automorphism 
$A^{(q)} := \left({X_3}, {X_2}, {\alpha}{X_1} + {\beta}{X_3}\right) \in {\mathop{\mathrm{Aff}}}_{3}\left(\mathbb{F}_{q}\right)$. 
It is easy to see that 
$A^{(q)} = \left({X_3}, {X_2}, {X_1}\right) \circ \left({X_1} + {\beta}{X_3}, {X_2}, {X_3}\right) \circ \left({\alpha}{X_1}, {X_2}, {X_3}\right) %
= T_{1, 3} \circ R_{1, 3}\left({\beta}\right) \circ D_{1}\left({\alpha}\right)$. 
We remark that ${N_T}\left(A\right) = {N_D}\left(A\right) = {N_R}\left(A\right) = 1$, ${\ell_A} = 3$, 
and 
$\mathop{\mathrm{sgn}}\left({\pi}_{q}\left(A^{(q)}\right)\right) %
= \mathop{\mathrm{sgn}}\left({\pi}_{q}\left(T_{1, 3}\right)\right) %
\times \mathop{\mathrm{sgn}}\left({\pi}_{q}\left(R_{1, 3}\left({\beta}\right)\right)\right) %
\times \mathop{\mathrm{sgn}}\left({\pi}_{q}\left(D_{1}\left({\alpha}\right)\right)\right)$. 
If $p = 2$ then by Equation~\eqref{eqn:perm_tame_main_aff_auto_sign_case_q_even_not_two} 
and by Equation~\eqref{eqn:perm_tame_main_aff_auto_sign_case_q_two_n_geq_3}, 
one can easily see that $\mathop{\mathrm{sgn}}\left({\pi}_{q}\left(A^{(q)}\right)\right) = 1$. 
If $q$ is odd then by Equation~\eqref{eqn:perm_tame_main_aff_auto_sign_case_q_odd_general}, 
\begin{equation*}
\mathop{\mathrm{sgn}}\left({\pi}_{q}\left(A^{(q)}\right)\right) = %
\left(-1\right)^{{\rm ord}_{\mathbb{F}_{q}^{*}}\left({\alpha}\right)} \times \left(-1\right)^{\frac{q-1}{2}}. 
\end{equation*}
In particular, if $q \equiv 1 \bmod 4$ then 
$\mathop{\mathrm{sgn}}\left({\pi}_{q}\left(A^{(q)}\right)\right) = \left(-1\right)^{{\rm ord}_{\mathbb{F}_{q}^{*}}\left({\alpha}\right)}$, 
and if $q \equiv 3 \bmod 4$ then 
$\mathop{\mathrm{sgn}}\left({\pi}_{q}\left(A^{(q)}\right)\right) = -1 \times \left(-1\right)^{{\rm ord}_{\mathbb{F}_{q}^{*}}\left({\alpha}\right)} %
= \left(-1\right)^{{\rm ord}_{\mathbb{F}_{q}^{*}}\left({\alpha}\right) + 1}$. 
}
\end{ex}
%===========================================================

%EOF

%%%%%%%%%%%%%%%%%%%%%%%%%%%%%%%%%%%%%%%%%%%%%%%%%%
\section{Sign of permutations induced by triangular automorphisms and tame automorphisms}
\label{sec:perm_tame_main_tame}
%%%%%%%%%%%%%%%%%%%%%%%%%%%%%%%%%%%%%%%%%%%%%%%%%%

In this section, we consider the sign of permutations induced by 
triangular automorphisms and tame automorphisms over finite fields. 
By Main Theorem~\ref{thm:perm_tame_main_elementary_auto_sign} 
and Main Theorem~\ref{thm:perm_tame_main_affine_auto_sign}, 
we obtain the following corollary (Corollary~\ref{cor:perm_tame_main_triangular_auto_sign}). 

%===========================================================
\begin{cor}
\label{cor:perm_tame_main_triangular_auto_sign}
{\rm\bfseries (Sign of triangular automorphisms)} \ 
Suppose that $J_{a, f}^{(q)}$ 
is a triangular automorphism over a finite field, namely, 
\begin{equation}
\begin{split}
J_{a, f}^{(q)} & = %
\left( {a_1}{X_1} + {f_1}({X_2}, \ldots, {X_n}), {a_2}{X_2} %
+ {f_2}({X_3}, \ldots, {X_n}), \ldots, {a_n}{X_n} + {f_n}\right) \in \mathop{\mathrm{BA}_{n}}(\mathbb{F}_{q}), \\
& 
{a_i} \in \mathbb{F}_{q} \ \left(i = 1, \ldots, n\right), \ {f_i} \in \mathbb{F}_{q}[{X_{i+1}}, \ldots, {X_n}] \ %
\left(i = 1, \ldots, n-1\right), \ {f_n} \in \mathbb{F}_{q}. %
\end{split}
\label{eqn:perm_tame_main_aff_triangular_poly_auto}
\end{equation}
If $q = 2^m$, $m \geq 2$ then ${\pi}_{q}\left(J_{a, f}^{(q)}\right) \in \mathop{\mathrm{Alt}}({\mathbb{F}_{q}^n})$. 
Namely, if $q = 2^m$, $m \geq 2$ then 
\begin{equation}
\mathop{\mathrm{sgn}}\left({\pi}_{q}\left(J_{a, f}^{(q)}\right)\right) = 1. %
\label{eqn:perm_tame_main_aff_triangular_auto_sign_q_even_not_two}
\end{equation}
If $q$ is odd then 
\begin{equation}
\mathop{\mathrm{sgn}}\left({\pi}_{q}\left(J_{a, f}^{(q)}\right)\right) %
= \left(-1\right)^{\sum_{i = 1}^{n} {\rm ord}_{\mathbb{F}_{q}^{*}}\left({a_i}\right)}. %
\label{eqn:perm_tame_main_aff_triangular_auto_sign_q_odd}
\end{equation}
In other words, if $q$ is odd then $\mathop{\mathrm{sgn}}\left({\pi}_{q}\left(J_{a, f}^{(q)}\right)\right)$ 
depends only on the coefficients ${a_1}, \ldots, {a_n}$. 
If $q = 2$ then 
\begin{equation}
\mathop{\mathrm{sgn}}\left({\pi}_{q}\left(J_{a, f}^{(q)}\right)\right) %
= \left(-1\right)^{M_{f_1}}, %
\label{eqn:perm_tame_main_aff_triangular_auto_sign_q_two}
\end{equation}
where $M_{f_1}$ is the number of monomials of the form 
$cX_2^{e_2}{\cdots}X_n^{e_n}$ 
with $c \in \mathbb{F}_{q}^{*}$ and ${e_2}, \ldots, {e_n} \geq 1$ 
appearing in the polynomial ${f_1} \in \mathbb{F}_{q}[{X_{2}}, \ldots, {X_n}]$. 
\end{cor}
%===========================================================

%===========================================================
\begin{proof}
By using the notation of Equation~\eqref{eqn:perm_tame_main_elementary_auto} 
and Equation~\eqref{eqn:perm_tame_main_aff_elementary_mat_d_lin_auto}, 
we have 
\begin{equation*}
J_{a, f}^{(q)} = D_{n}\left({a_n}\right) \circ \cdots D_{1}\left({a_1}\right) \circ E_{{a_1^{-1}}{f_{1}}}^{(q)} \circ \cdots \circ E_{{a_n^{-1}}{f_{n}}}^{(q)}. %
\end{equation*}
Hence we obtain 
\begin{equation}
\mathop{\mathrm{sgn}}\left({\pi}_{q}\left(J_{a, f}^{(q)}\right)\right) %
= \prod_{i=1}^{n} \mathop{\mathrm{sgn}}\left({\pi}_{q}\left(D_{i}\left({a_i}\right)\right)\right) 
\times \prod_{i=1}^{n} \mathop{\mathrm{sgn}}\left({\pi}_{q}\left(E_{{a_i^{-1}}{f_{i}}}^{(q)}\right)\right). 
\label{eqn:perm_tame_main_aff_triangular_auto_decomposition_aff_and_elementary_sign}
\end{equation}
By Equation~\eqref{eqn:perm_tame_main_aff_triangular_auto_decomposition_aff_and_elementary_sign}, 
Main Theorem~\ref{thm:perm_tame_main_elementary_auto_sign}, 
and Main Theorem~\ref{thm:perm_tame_main_affine_auto_sign}, 
we obtain the desired results. 
\end{proof}
%===========================================================

%===========================================================
\begin{cor}
\label{cor:perm_tame_main_strictly_triangular_auto_sign}
{\rm\bfseries (Sign of strictly triangular automorphisms)} \ 
Suppose that $J_{f}^{(q)}$ 
is a strictly triangular automorphism over a finite field, namely, 
\begin{equation}
\begin{split}
J_{f}^{(q)} & = %
\left({X_1} + {f_1}({X_2}, \ldots, {X_n}), {X_2} %
+ {f_2}({X_3}, \ldots, {X_n}), \ldots, {X_n} + {f_n}\right) \in \mathop{\mathrm{BA}_{n}}(\mathbb{F}_{q}), \\
& 
{f_i} \in \mathbb{F}_{q}[{X_{i+1}}, \ldots, {X_n}] \ %
\left(i = 1, \ldots, n-1\right), \ {f_n} \in \mathbb{F}_{q}. %
\end{split}
\label{eqn:perm_tame_main_aff_strictry_triangular_poly_auto}
\end{equation}
If $q$ is odd or $q = 2^m$, $m \geq 2$ then ${\pi}_{q}\left(J_{f}^{(q)}\right) \in \mathop{\mathrm{Alt}}({\mathbb{F}_{q}^n})$. 
Namely, if $q$ is odd or $q = 2^m$, $m \geq 2$ then 
\begin{equation}
\mathop{\mathrm{sgn}}\left({\pi}_{q}\left(J_{f}^{(q)}\right)\right) = 1. %
\label{eqn:perm_tame_main_aff_strictly_triangular_auto_sign_q_not_two}
\end{equation}
If $q = 2$ then 
\begin{equation}
\mathop{\mathrm{sgn}}\left({\pi}_{q}\left(J_{f}^{(q)}\right)\right) %
= \left(-1\right)^{M_{f_1}}, %
\label{eqn:perm_tame_main_aff_strictly_triangular_auto_sign_q_two}
\end{equation}
where $M_{f_1}$ is the number of monomials of the form 
$cX_2^{e_2}{\cdots}X_n^{e_n}$ 
with $c \in \mathbb{F}_{q}^{*}$ and ${e_2}, \ldots, {e_n} \geq 1$ 
appearing in the polynomial ${f_1} \in \mathbb{F}_{q}[{X_{2}}, \ldots, {X_n}]$. 
\end{cor}
%===========================================================

%===========================================================
\begin{proof}
It follows immediately from Corollary~\ref{cor:perm_tame_main_triangular_auto_sign}. 
\end{proof}
%===========================================================

\medskip

We recall that for any ${\phi}^{(q)} \in \mathop{\mathrm{TA}_{n}}(\mathbb{F}_{q})$, 
there exist $l \in \mathbb{Z}_{\geq 0}$, ${\epsilon_1}, {\epsilon_2} \in \{0, 1\} \subset \mathbb{Z}$, 
$\widetilde{A_{s, \mathbf{b}^{(s)}}^{(q)}} \in \mathop{\mathrm{Aff}}_{n}(\mathbb{F}_{q})$ ($1 \leq s \leq l + 1$) of the form 
\begin{equation*}
\widetilde{A_{s, \mathbf{b}^{(s)}}^{(q)}} = %
\left(\left(\sum_{i=1}^{n} {a_{1, i}^{(s)}}{X_i}\right) + {b_1^{(s)}}, \ldots, \left(\sum_{i=1}^{n} {a_{n, i}^{(s)}}{X_i}\right) + {b_n^{(s)}}\right), %
\end{equation*}
and 
$J_{s, {t^{(s)}}, {f^{(s)}}}^{(q)} \in \mathop{\mathrm{BA}_{n}}(\mathbb{F}_{q})$ ($1 \leq s \leq l$) of the form 
\begin{equation*}
\begin{split}
J_{s, {t^{(s)}}, {f^{(s)}}}^{(q)} & = %
\left( {t_1^{(s)}}{X_1} + {f_1^{(s)}}({X_2}, \ldots, {X_n}), \ldots, {t_n^{(s)}}{X_n} + {f_n^{(s)}}\right), %
\end{split}
\end{equation*}
such that 
\begin{equation}
{\phi}^{(q)} = {\widetilde{A_{1, \mathbf{b}^{(1)}}^{(q)}}}^{\epsilon_1} \circ J_{1, {t^{(1)}}, {f^{(1)}}}^{(q)} %
\circ \cdots \circ %
{\widetilde{A_{l, \mathbf{b}^{(l)}}^{(q)}}} \circ J_{l, {t^{(l)}}, {f^{(l)}}}^{(q)} %
\circ {\widetilde{A_{l+1, \mathbf{b}^{(l+1)}}^{(q)}}}^{\epsilon_2}, %
\label{eqn:perm_tame_main_tame_decomposition}
\end{equation}
$\widetilde{A_{s, \mathbf{b}^{(s)}}^{(q)}} \not\in \mathop{\mathrm{BA}_{n}}(\mathbb{F}_{q})$ for $2 \leq s \leq l + 1$, 
and $J_{s, {t^{(s)}}, {f^{(s)}}}^{(q)} \not\in \mathop{\mathrm{Aff}}_{n}(\mathbb{F}_{q})$ for $1 \leq s \leq l$ 
(for example, \cite[Lemma~5.1.1]{Ess00}). 

We use the symbol $A_{s}^{(q)}$ to denote the homogeneous part (linear automorphism) 
of the affine automorphism $\widetilde{A_{s, \mathbf{b}^{(s)}}^{(q)}}$ 
(as in Equation~\eqref{eqn:perm_tame_main_aff_auto} and Equation~\eqref{eqn:perm_tame_main_lin_auto}) for $1 \leq s \leq l + 1$. 
Furthermore, for each $s$ ($1 \leq s \leq l$), 
we denote by $M_{f_{1}^{(s)}}$ the number of monomials of the form 
$cX_2^{e_2}{\cdots}X_n^{e_n}$ 
with $c \in \mathbb{F}_{q}^{*}$ and ${e_2}, \ldots, {e_n} \geq 1$ 
appearing in the polynomial ${f_{1}^{(s)}} \in \mathbb{F}_{q}[{X_{2}}, \ldots, {X_n}]$. 

\medskip

The following corollary (Corollary~\ref{cor:perm_tame_main_tame_auto_sign}) 
states that if we know Equation~\eqref{eqn:perm_tame_main_tame_decomposition} 
for a given ${\phi}^{(q)} \in \mathop{\mathrm{TA}_{n}}(\mathbb{F}_{q})$, 
then one can easily compute the sign of the permutation induced 
by ${\phi}^{(q)} \in \mathop{\mathrm{TA}_{n}}(\mathbb{F}_{q})$. 

%===========================================================
\begin{cor}
\label{cor:perm_tame_main_tame_auto_sign}
{\rm\bfseries (Sign of tame automorphisms)} \ 
With notation as above, the following assertions are hold: 

\noindent
{\normalfont (1)}\ 
If $q = {2^m}$, $m \geq 2$ then 
\begin{equation}
\mathop{\mathrm{sgn}}\left({\pi}_{q}\left({\phi}^{(q)}\right)\right) = 1. 
\label{eqn:perm_tame_main_tame_auto_sign_case_q_even_not_two}
\end{equation}

\noindent
{\normalfont (2)}\ If $q = 2$ and $n = 2$ then 
\begin{align}
\mathop{\mathrm{sgn}}\left({\pi}_{q}\left({\phi}^{(q)}\right)\right) %
& = %
\left(-1\right)^{{\epsilon_1}{N_T}\left(A_{1}\right) + \left(\sum_{s = 2}^{l} {N_T}\left(A_{s}\right)\right) + {\epsilon_2}{N_T}\left(A_{l+1}\right) + \sum_{s=1}^{\ell} M_{f_{1}^{(s)}} }. %
\label{eqn:perm_tame_main_tame_auto_sign_case_q_two_n_two}
\end{align}

\noindent
{\normalfont (3)}\ If $q = 2$ and $n \geq 3$ then 
\begin{align}
\mathop{\mathrm{sgn}}\left({\pi}_{q}\left({\phi}^{(q)}\right)\right) %
& = %
\left(-1\right)^{\sum_{s=1}^{\ell} M_{f_{1}^{(s)}} }. %
\label{eqn:perm_tame_main_tame_auto_sign_case_q_two_n_geq_3}
\end{align}

\noindent
{\normalfont (4)}\ 
If $q$ is odd then 
\begin{align}
\mathop{\mathrm{sgn}}\left({\pi}_{q}\left({\phi}^{(q)}\right)\right) %
& = %
\left(-1\right)^{
\sum_{s = 2}^{l} \left(\sum_{j = 1}^{{N_D}(A_{s})} {\rm ord}_{\mathbb{F}_{q}^{*}}\left({c_j^{(A_{s})}}\right) \right)} \nonumber\\%
& \hspace{0.5cm}%
\times 
\left(-1\right)^{
{\epsilon_1} \left(\sum_{j = 1}^{{N_D}(A_{1})} {\rm ord}_{\mathbb{F}_{q}^{*}}\left({c_j^{(A_{1})}}\right)\right)} \nonumber\\%
& \hspace{0.5cm}%
\times 
\left(-1\right)^{
{\epsilon_2} \left(\sum_{j = 1}^{{N_D}(A_{l+1})} {\rm ord}_{\mathbb{F}_{q}^{*}}\left({c_j^{(A_{l+1})}}\right)\right)} \nonumber\\%
& \hspace{0.5cm}%
\times 
\left(-1\right)^{
\frac{q-1}{2}\left({\epsilon_1}{{N_T}(A_{1})} + {\epsilon_2}{{N_T}(A_{l+1})} + \sum_{s = 2}^{l} {{N_T}(A_{s})} \right)} \nonumber\\%
& \hspace{0.5cm}%
\times 
\left(-1\right)^{
\sum_{\stackrel{1 \leq i \leq n}{1 \leq s \leq l}} 
{\rm ord}_{\mathbb{F}_{q}^{*}}\left({t_i^{(s)}}\right) %
}. 
\label{eqn:perm_tame_main_tame_auto_sign_case_q_odd_general}
\end{align}
In particular, if $q \equiv 1 \bmod 4$ then 
\begin{align}
\mathop{\mathrm{sgn}}\left({\pi}_{q}\left({\phi}^{(q)}\right)\right) %
& = %
\left(-1\right)^{
\sum_{s = 2}^{l} \left(\sum_{j = 1}^{{N_D}(A_{s})} {\rm ord}_{\mathbb{F}_{q}^{*}}\left({c_j^{(A_{s})}}\right) \right)} \nonumber\\%
& \hspace{0.5cm}%
\times 
\left(-1\right)^{
{\epsilon_1} \left(\sum_{j = 1}^{{N_D}(A_{1})} {\rm ord}_{\mathbb{F}_{q}^{*}}\left({c_j^{(A_{1})}}\right)\right)} \nonumber\\%
& \hspace{0.5cm}%
\times 
\left(-1\right)^{
{\epsilon_2} \left(\sum_{j = 1}^{{N_D}(A_{l+1})} {\rm ord}_{\mathbb{F}_{q}^{*}}\left({c_j^{(A_{l+1})}}\right)\right)} \nonumber\\%
& \hspace{0.5cm}%
\times 
\left(-1\right)^{
\sum_{\stackrel{1 \leq i \leq n}{1 \leq s \leq l}} 
{\rm ord}_{\mathbb{F}_{q}^{*}}\left({t_i^{(s)}}\right) %
}. 
\label{eqn:perm_tame_main_tame_auto_sign_case_q_odd_1_mod_4}
\end{align}
\end{cor}
%===========================================================

%===========================================================
\begin{proof}
By Equation~\eqref{eqn:perm_tame_main_tame_decomposition}, 
we obtain 
\begin{align}
\mathop{\mathrm{sgn}}\left({\pi}_{q}\left({\phi}^{(q)}\right)\right) %
& = %
\prod_{i=2}^{l} \mathop{\mathrm{sgn}}\left({\pi}_{q}\left(A_{i}^{(q)}\right)\right) 
\times \left(\mathop{\mathrm{sgn}}\left({\pi}_{q}\left(A_{1}^{(q)}\right)\right)\right)^{\epsilon_1} \nonumber\\%
& \hspace{0.5cm}%
\times \left(\mathop{\mathrm{sgn}}\left({\pi}_{q}\left(A_{l+1}^{(q)}\right)\right)\right)^{\epsilon_2} 
\times \prod_{i=1}^{l} \mathop{\mathrm{sgn}}\left({\pi}_{q}\left(J_{i, {t^{(i)}}, {f^{(i)}}}^{(q)}\right)\right). 
\label{eqn:perm_tame_main_tame_decomposition_sign}
\end{align}

By Equation~\eqref{eqn:perm_tame_main_tame_decomposition_sign}, 
Main Theorem~\ref{thm:perm_tame_main_affine_auto_sign}, 
and Corollary~\ref{cor:perm_tame_main_triangular_auto_sign}, 
we obtain the desired results. 
\end{proof}
%===========================================================

%===========================================================
\begin{rem}
\label{rem:perm_tame_main_tame_auto_sign_uniqueness}
{\rm 
As in Remark~\ref{rem:perm_tame_main_affine_auto_sign_uniqueness}, 
one can see that 
Equation~\eqref{eqn:perm_tame_main_tame_auto_sign_case_q_two_n_two}, 
\eqref{eqn:perm_tame_main_tame_auto_sign_case_q_two_n_geq_3}, 
\eqref{eqn:perm_tame_main_tame_auto_sign_case_q_odd_general}, 
and \eqref{eqn:perm_tame_main_tame_auto_sign_case_q_odd_1_mod_4} 
depend on the representations~\eqref{eqn:perm_tame_main_aff_auto_elementary_matrices_decomposition} 
and \eqref{eqn:perm_tame_main_tame_decomposition} which are not unique. 
However, since ${\pi}_{q}\left({\phi}^{(q)}\right)$ is uniquely determined as a permutation on $\mathbb{F}_{q}^{n}$, 
$\mathop{\mathrm{sgn}}\left({\pi}_{q}\left({\phi}^{(q)}\right)\right)$ does not depend on 
the representations~\eqref{eqn:perm_tame_main_aff_auto_elementary_matrices_decomposition} 
and \eqref{eqn:perm_tame_main_tame_decomposition}. 
}
\end{rem}
%===========================================================

%===========================================================
\begin{rem}
\label{rem:perm_tame_main_tame_auto_Maubach_theorem_comparison}
{\rm 
Suppose that $n$ is greater than or equal to two. 
Corollary~\ref{cor:perm_tame_main_tame_auto_sign} tells us that 
${\pi}_{q}\left(\mathop{\mathrm{TA}_{n}}\left(\mathbb{F}_{q}\right)\right) \subset \mathop{\mathrm{Sym}}\left({\mathbb{F}_{q}^n}\right)$ 
(this is a trivial inclusion relation) if $q$ is odd or $q = 2$, and 
${\pi}_{q}\left(\mathop{\mathrm{TA}_{n}}\left(\mathbb{F}_{q}\right)\right) \subset \mathop{\mathrm{Alt}}\left({\mathbb{F}_{q}^n}\right)$ 
if $q = {2^m}$ and $m \geq 2$. 
This indicates that 
Corollary~\ref{cor:perm_tame_main_tame_auto_sign} is strictly weaker than \cite[Theorem~2.3]{Mau01}. 
However, 
Main Theorem~\ref{thm:perm_tame_main_elementary_auto_sign}, 
Main Theorem~\ref{thm:perm_tame_main_affine_auto_sign}, 
Corollary~\ref{cor:perm_tame_main_triangular_auto_sign}, 
and Corollary~\ref{cor:perm_tame_main_tame_auto_sign} are useful 
for determining the sign of permutations induced by tame automorphisms over finite fields. 
The reasons are as follows. 
Firstly, for $q = {2^m}$ and $m \geq 2$, we prove that 
${\pi}_{q}\left(\mathop{\mathrm{TA}_{n}}(\mathbb{F}_{q})\right) \subset \mathop{\mathrm{Alt}}({\mathbb{F}_{q}^n})$ 
by directly showing that ${\pi}_{q}\left({\mathop{\mathrm{EA}}}_{n}\left(\mathbb{F}_{q}\right)\right) \subset \mathop{\mathrm{Alt}}({\mathbb{F}_{q}^n})$ 
and ${\pi}_{q}\left({\mathop{\mathrm{Aff}}}_{n}\left(\mathbb{F}_{q}\right)\right) \subset \mathop{\mathrm{Alt}}({\mathbb{F}_{q}^n})$. 
Secondly, if $q$ is odd then one can not determine the sign of permutation induced by an elementary automorphisms 
over a finite field by using \cite[Theorem~2.3]{Mau01}. 
In contrast to \cite[Theorem~2.3]{Mau01}, Main Theorem~\ref{thm:perm_tame_main_elementary_auto_sign} 
tells us that if $q$ is odd then each permutation induced by an elementary automorphism over a finite field is even. 
In other words, if $q$ is odd then ${\pi}_{q}\left({\mathop{\mathrm{EA}}}_{n}\left(\mathbb{F}_{q}\right)\right) \subset \mathop{\mathrm{Alt}}({\mathbb{F}_{q}^n})$ 
(Corollary~\ref{cor:perm_tame_main_elementary_auto_sign}). 
Similarly, in contrast to \cite[Theorem~2.3]{Mau01}, 
Main Theorem~\ref{thm:perm_tame_main_affine_auto_sign} tells us that 
if $q = 2$ and $n \geq 3$ then each permutation induced by an affine automorphism over a finite field is even. 
Namely, if $q = 2$ and $n \geq 3$ then 
${\pi}_{q}\left({\mathop{\mathrm{Aff}}}_{n}\left(\mathbb{F}_{q}\right)\right) \subset \mathop{\mathrm{Alt}}({\mathbb{F}_{q}^n})$ 
(Corollary~\ref{cor:perm_tame_main_affine_auto_sign}). 
Thus, our results (Main Theorem~\ref{thm:perm_tame_main_elementary_auto_sign}, Main Theorem~\ref{thm:perm_tame_main_affine_auto_sign}, 
Corollary~\ref{cor:perm_tame_main_triangular_auto_sign}, and Corollary~\ref{cor:perm_tame_main_tame_auto_sign}) 
and \cite[Theorem~2.3]{Mau01} are complementary to each other. 
}
\end{rem}
%===========================================================

%EOF

%\input{perm_tame_conclusion.tex}
%%%%%%%%%%%%%%%%%%%%%%%%%%%%%%%%%%%%%%%%%%%%%%%%%%%%%%%%%%%%%%%%%%%%%%
\section*{Acknowledgements}
%%%%%%%%%%%%%%%%%%%%%%%%%%%%%%%%%%%%%%%%%%%%%%%%%%%%%%%%%%%%%%%%%%%%%%

This work was supported by JSPS KAKENHI Grant-in-Aid for Young Scientists (B) 16K16066. 

%EOF

%%%%%%%%%%%%%%%%%%%%%%%%%%%%%%%%%%%%%%%%%%%%%%%%%%

%%%%%%%%%%%%%%%%%%%%%%%%%%%%%%%%%%%%%%%%%%%%%%%%%%

\end{document}